\documentclass[12pt]{amsart}
\usepackage{amsmath,amssymb}

\allowdisplaybreaks[1]

\title{On certain mean values of the double zeta-function}

\author{Soichi Ikeda}
\author{Kaneaki Matsuoka}
\author{Yoshikazu Nagata}

\address{Graduate School of Mathematics, Nagoya University,
Furocho, Chikusaku, Nagoya 464-8602, Japan}

\email{m10004u@math.nagoya-u.ac.jp}
\email{m10041v@math.nagoya-u.ac.jp}
\email{m10035y@math.nagoya-u.ac.jp}



\theoremstyle{plain}
\newtheorem{theorem}{Theorem}
\newtheorem{lemma}{Lemma}
\newtheorem{corollary}{Corollary}

\theoremstyle{remark}
\newtheorem{remark}{Remark}

\numberwithin{theorem}{section}
\numberwithin{lemma}{section}
\numberwithin{corollary}{section}
\numberwithin{remark}{section}
\numberwithin{equation}{section}

\begin{document}
\begin{abstract}
In this paper we discuss three types of mean values
of the Euler double zeta function. In order to get results 
we introduce three approximate formulas for this function.
\end{abstract}

\maketitle

\section{Introduction}
Let $s_1 = \sigma_1 + it_1$ and $s_2 = \sigma_2 + it_2$ with
$\sigma_1, \sigma_2, t_1, t_2 \in \mathbb{R}$. The Euler double
zeta-function is defined by
\begin{equation}  \label{eq_db_zeta}
  \zeta_2(s_1, s_2) = \sum_{m=1}^{\infty} \frac{1}{m^{s_1}}
    \sum_{n=1}^{\infty} \frac{1}{(m+n)^{s_2}}.
\end{equation}
This series is absolutely convergent for $\sigma_2 > 1$ and
$\sigma_1 + \sigma_2 > 2$ \cite{matsumoto}.
We can continue $\zeta_2(s_1,s_2)$ meromorphically
to $\mathbb{C}^2$, which is holomorphic in
\[ \{ (s_1,s_2) \in \mathbb{C}^2 \mid  s_2 \neq 1,
s_1 + s_2 \notin \{ 2, 1, 0, -2, -4, -6, \dots \} \} \]
as was proved in \cite{aki_ega_tani}.
(The first study of the analytic continuation
of $\zeta_2(s_1, s_2)$ is Atkinson's work in \cite{atkinson}.
Akiyama, Egami and Tanigawa studied the analytic continuation
of not only $\zeta_2(s_1, s_2)$ but also more general multiple
zeta-functions in \cite{aki_ega_tani}.
Zhao also obtained the continuation in \cite{zhao} independently.)

The analytic properties of $\zeta_2(s_1,s_2)$ were studied by
various authors (for example, Kiuchi-Tanigawa-Zhai \cite{kiu_tani_zhai},
Matsumoto \cite{matsumoto}, \cite{matfeq}). Recently, Matsumoto
and Tsumura studied the mean values
\begin{equation} \label{mt_intt2}
  \int_2^T |\zeta_2(s_1, s_2)|^2 dt_2,
\end{equation}
where $s_1$ is a fixed complex number. This is the first study
of the mean values of $\zeta_2(s_1,s_2)$. In this paper we study
(\ref{mt_intt2}) in the regions which are not covered in the
work of Matsumoto and Tsumura and introduce new types of
mean values of $\zeta_2(s_1,s_2)$. This paper is inspired by
Matsumoto and Tsumura \cite{matsu_tsumu}.

In this paper we prove the following theorems.

\begin{theorem}  \label{th_intt1}
Let $s_1 = \sigma_1 + it_1, s_2 = \sigma_2 + it_2 \in \mathbb{C}$,
$T \ge 2$ and
\[ I^{[1]}(T) = \int_2^T |\zeta_2(s_1, s_2)|^2 dt_1. \]
Assume that when $t_1$ moves from $2$ to $T$,
the point $(s_1, s_2) \in \mathbb{C}^2$ does not
encounter the singularities of $\zeta_2(s_1, s_2)$.
In the case $\sigma_1 + \sigma_2 > 2$, we have
\[ I^{[1]}(T) = \zeta_2^{[1]}(2 \sigma_1, s_2) T + O(1), \]
where the implied constant depends on $s_2$ and
$\zeta_2^{[1]}(2 \sigma_1, s_2)$ is a series which converges
$\sigma_1 + \sigma_2 > 3/2$ (we define $\zeta_2^{[1]}(\sigma_1, s_2)$
in the next setcion).
In the case $3/2 < \sigma_1 + \sigma_2 \le 2$, we have
\[ I^{[1]}(T) = \zeta_2^{[1]}(2 \sigma_1, s_2) T +
     \begin{cases}
       O(T^{4 - 2 \sigma_1 - 2 \sigma_2}) &
         (3/2 < \sigma_1 + \sigma_2 < 2), \\
       O((\log T)^2) & ( \sigma_1 + \sigma_2 = 2).
     \end{cases} \]
In the case $\sigma_1 + \sigma_2 = 3/2$, we have
\[ I^{[1]}(T) =| s_2 - 1 |^{-2}T \log T + O(T).\]
\end{theorem}

\begin{theorem}  \label{th_intt2}
Let $s_1 = \sigma_1 + it_1, s_2 = \sigma_2 + it_2 \in \mathbb{C}$,
$T \ge 2$ and
\[ I^{[2]}(T) = \int_{2}^{T} | \zeta_2(s_1,s_2) |^2 dt_2. \]
Assume that when $t_2$ moves from $2$ to $T$,
the point $(s_1, s_2) \in \mathbb{C}^2$ does not
encounter the singularities of $\zeta_2(s_1, s_2)$.
In the case $\sigma_2 > 1$ and $\sigma_1 + \sigma_2 > 2$, we have
\[ I^{[2]}(T) = \zeta_2^{[2]}(s_1, 2 \sigma_2) T + O(1), \]
where the implied constant depends on $s_1$ and
$\zeta_2^{[2]}(s_1, 2 \sigma_2)$ is a series which converges
$\sigma_1+\sigma_2 > 3/2$ and $\sigma_2 > 1/2$
($\zeta_2^{[2]}(s_1, \sigma_2)$ is used in \cite{matsu_tsumu}
and we show the definition of $\zeta_2^{[2]}(s_1, \sigma_2)$
in the next section).
In the case $\sigma_1 > 1$ and $1/2 < \sigma_2 \le 1$,
we have
\[ I^{[2]}(T) = \zeta_2^{[2]}(s_1, 2 \sigma_2) T +
     \begin{cases}
       O(T^{2 - 2 \sigma_2} )  & (\sigma_2 \neq 1), \\
       O((\log T)^2)  & (\sigma_2 = 1).
     \end{cases} \]
In the case $\sigma_1 \le 1$, $3/2 < \sigma_1 + \sigma_2 \le 2$
and $s_1 \neq 1$, we have
\[ I^{[2]}(T) = \zeta_2^{[2]}(s_1, 2 \sigma_2) T +
     \begin{cases}
       O(T^{4 - 2 \sigma_1 - 2 \sigma_2} ) &
         ( \sigma_1 + \sigma_2 \neq 2), \\
       O((\log T)^2) &
         (\sigma_1 + \sigma_2 = 2).
     \end{cases} \]
In the case $s_1 = 1$ and $1/2 < \sigma_2 \le 1$, we have
\[ I^{[2]}(T) = \zeta_2^{[2]}(s_1, 2 \sigma_2) T +
     \begin{cases}
       O(T^{2 - 2 \sigma_2} (\log T)^2) & (\sigma_2 \neq 1), \\
       O((\log T)^4) & (\sigma_2 = 1).
     \end{cases} \]
In the case $\sigma_1 > 1$ and $\sigma_2 = 1/2$, we have
\[ I^{[2]}(T) = |\zeta(s_1)|^2 T \log T + O(T). \]
In the case $\sigma_1 + \sigma_2 = 3/2$ and $\sigma_2 > 1/2$,
we have
\[ I^{[2]}(T) = |s_1 - 1|^{-2} T \log T + O(T). \]
In the case $\sigma_2 = 1/2$, $\sigma_1 = 1$ and $s_1 \neq 1$,
we have
\[ I^{[2]}(T) = ( |s_1 - 1|^{-2} + |\zeta(s_1)|^2 ) T \log T + O(T).\]
In the case $\sigma_2 = 1/2$ and $s_1 = 1$, we have
\[ I^{[2]}(T) = \frac{T (\log T)^3}{3} + O(T (\log T)^2). \]
\end{theorem}

\begin{theorem}  \label{th_intt}
Let $s_1 = \sigma_1 + it, s_2 = \sigma_2 + it \in \mathbb{C}$,
$T \ge 2$ and
\[ I^{\Box}(T) = \int_2^T |\zeta_2(s_1, s_2)|^2 dt. \]
In the case $\sigma_2 > 1$ and $\sigma_1 + \sigma_2 > 2$, we have
\[ I^{\Box}(T) = \zeta_2^{\Box} (\sigma_1, \sigma_2) T + O(1), \]
where $\zeta_2^{\Box}(\sigma_1, \sigma_2)$ is a series which converges
if and only if $\sigma_2 > 1/2$ and $\sigma_1 + \sigma_2 > 1$
(we define $\zeta_2^{\Box}(\sigma_1, \sigma_2)$ in the next section).
In the case $\sigma_1 > 1$ and $1/2 < \sigma_2 \le 1$, we have
\[ I^{\Box}(T) = \zeta_2^{\Box} (\sigma_1, \sigma_2) T +
   O(T^{2 - 2 \sigma_2 + \epsilon}) + O(T^{1/2}) \]
for sufficiently small $\epsilon > 0$. In the case
$\sigma_1 \le 1$ and $3/2 < \sigma_1 + \sigma_2 \le 2$, we have
\[ I^{\Box}(T) = \zeta_2^{\Box} (\sigma_1, \sigma_2) T +
   O(T^{4 - 2 \sigma_1 - 2 \sigma_2 + \epsilon}) +
   O(T^{1/2}) \]
for sufficiently small $\epsilon > 0$.
In the case $\sigma_1 > 1$ and $\sigma_2 = 1/2$, we have
\[ I^{\Box}(T) \asymp T \log T. \]
\end{theorem}

Matsumoto and Tsumura introduced $I^{[2]}(T)$ and 
studied the cases
\begin{enumerate}
\item $\sigma_1 > 1$ and $\sigma_2 > 1$ (Theorem 1.1 of \cite{matsu_tsumu}),
\item $\sigma_1 + \sigma_2 > 2$ and $1/2 < \sigma_2 \le 1$
(Theorem 1.2 of \cite{matsu_tsumu}),
\item $1/2 < \sigma_1 < 3/2$, $1/2 < \sigma_2 \le 1$ and
$3/2 < \sigma_1 + \sigma_2 \le 2$ (Theorem 1.3 of \cite{matsu_tsumu}).
\end{enumerate}
They conjectured that when $\sigma_1 + \sigma_2 = 3/2$, the form
of the main term of the mean square formula would not be $CT$
(with a constant $C$; most probably, some log-factor would
appear)(see their conjecture (ii) in \cite{matsu_tsumu}).
Our results include the regions which Matsumoto and Tsumura
did not study and give an improvement on the error
estimate. Moreover by Theorem 1.2 we see that their conjecture
(ii) is true.

Outlines of the proof of our theorems are as follows.
We can obtain Theorem \ref{th_intt1} and Theorem \ref{th_intt2}
by using the mean value theorems for Dirichlet polynomials
and suitable approximate formulas in  each theorem
(cf. Theorem 3.1 and Theorem 6.3 in Matsumoto and Tsumura
\cite{matsu_tsumu}).
The approximate formulas used in the proof of Theorem \ref{th_intt1}
and Theorem \ref{th_intt2} are derived from the Euler-Maclaurin
formula and the simplest approximate formula to $\zeta(s)$
due to Hardy and Littlewood. On the other hand
we need a more elaborate method to get the proof of Theorem \ref{th_intt}.
In order to obtain the suitable approximate formula for
$\zeta_{2}(\sigma_{1}+it, \sigma_{2}+it)$
we need the technique of Kiuchi and Tanigawa \cite{kiu_tani}, which enables us
to get good estimates of the error terms in the Euler-Maclaurin formula.

In Theorem \ref{th_intt1} (resp. Theorem \ref{th_intt2})
we regard $s_{2}$ (resp. $s_{1}$) as a constant term.
On the other hand, from the study of Kiuchi, Tanigawa and Zhai
\cite{kiu_tani_zhai}, we know that the behavior of
$|\zeta_2(s_1, s_2)|$ depends on both $s_1$ and $s_2$ strongly.
Therefore it is also important to consider a mean value
which depends on both $s_1$ and $s_2$.

From Theorem \ref{th_intt1} and Theorem \ref{th_intt2} we may expect that
the behavior of $\zeta_{2}(s_{1},s_{2})$ in the region
$\sigma_{1}+\sigma_{2}=3/2$
is special (Matsumoto and Tsumura conjectured that
$\sigma_{1}+\sigma_{2}=3/2$ might be the double analogue of
the critical line of the Riemann zeta-function
(see Remark 1.6 in \cite{matsu_tsumu})).
The error terms in Theorem \ref{th_intt} support their conjecture.
However, we can take a different point of view.
For the Riemann zeta function $\zeta(\sigma + it)$, we know that
\[ \int_2^T |\zeta(\sigma + it)|^2 dt \sim \zeta(2 \sigma) T \]
for $\sigma > 1/2$ and
\[ \int_2^T |\zeta(1/2 + it)|^2 dt \sim T \log T \]
hold (see, for example, Theorem 7.2 and Theorem 7.3 in \cite{titchmarsh}).
The line $\sigma = 1/2$ is the critical line
for $\zeta(\sigma + it)$ and the series
\[ \zeta(2 \sigma) = \sum_{n=1}^{\infty} \frac{1}{n^{2 \sigma}} \]
diverges on $\sigma = 1/2$. On the other hand,
$\zeta_2^{\Box}(\sigma_1, \sigma_2)$
converges if and only if $\sigma_2 > 1/2$ and $\sigma_1 + \sigma_2 > 1$.
Moreover, if $\sigma_1 = \sigma_2 > 1/2$ then
$I^{\Box}(T) \sim \zeta_2^{\Box}(\sigma_1, \sigma_2) T$ holds
by
\[ \int_2^T |\zeta(\sigma + it)|^4 dt = O(T) \]
for $\sigma > 1/2$ (see Theorem 7.5 in \cite{titchmarsh})
and Carlson's mean value theorem (see p. 304 in \cite{tit_func}).
Hence we can expect that
$I^{\Box}(T) \sim \zeta_2^{\Box}(\sigma_1, \sigma_2) T$ holds
for $\sigma_2 > 1/2$ and $\sigma_1 + \sigma_2 > 1$ and
the boundary of the region $\sigma_2 > 1/2$
and $\sigma_1 + \sigma_2 > 1$ is an analogue of the critical line
for $\zeta_2(\sigma_1 + it, \sigma_2 + it)$.

\section{Lemmas for the proof of Theorems}
In this section, we collect some auxiliary results and definitions.

First, we give the definition of $\zeta_2^{[1]}(\sigma_1, s_2)$,
$\zeta_2^{[2]}(s_1, \sigma_2)$ and $\zeta_2^{\Box}(\sigma_1, \sigma_2)$.

We define
\[ \zeta_2^{[1]}(\sigma_1, s_2) = \sum_{m=1}^{\infty}
     \frac{1}{m^{\sigma_1}}
     \biggl| \zeta(s_2) - \sum_{n=1}^{m} \frac{1}{n^{s_2}} \biggr|^2 \]
for $s_2 \neq 1$. Since we have
\begin{equation} \label{eval_zeta21}
\zeta_2^{[1]}(2 \sigma_1, s_2) \ll \sum_{m=1}^{\infty}
\begin{cases}
  m^{2 - 2 \sigma_1 - 2 \sigma_2} & (\sigma_2 > 1) \\
  m^{-2 \sigma_1} (\log m)^2 & (\sigma_2 = 1) \\
  m^{2 - 2 \sigma_1 - 2 \sigma_2} & (\sigma_2 < 1),
\end{cases}
\end{equation}
the series $\zeta_2^{[1]}(2 \sigma_1, s_2)$
converges in the region $\sigma_1 + \sigma_2 > 3/2$.

We define
\[ \zeta_2^{[2]}(s_1, \sigma_2) = \sum_{n = 2}^{\infty}
     \biggl| \sum_{m = 1}^{n - 1} \frac{1}{m^{s_1}} \biggr|^2
     \frac{1}{n^{\sigma_2}} \]
(this definition is the same as \cite{matsu_tsumu}).
Since we have
\begin{equation} \label{eval_zeta22}
\zeta_2^{[2]}(s_1, 2 \sigma_2) \ll \sum_{n=2}^{\infty}
\begin{cases}
  n^{-2 \sigma_2} & (\sigma_1 > 1) \\
  n^{-2 \sigma_2} (\log n)^2 & (\sigma_1 = 1) \\
  n^{2 - 2 \sigma_1 - 2 \sigma_2} & (\sigma_1 < 1),
\end{cases}
\end{equation}
the series $\zeta_2^{[2]}(s_1, 2 \sigma_2)$
converges in the region $\sigma_2 > 1/2$ and
$\sigma_1 + \sigma_2 > 3/2$.

We define
\[ \zeta_2^{\Box} (\sigma_1, \sigma_2) = \sum_{k=2}^{\infty}
     \Biggl( \sum_{mn=k \atop m < n} \frac{1}{m^{\sigma_1} n^{\sigma_2}}
     \Biggr)^2. \]
We note that $\#\{(m,n)| mn=k, m<n\}\ll k^{\epsilon}$ for any
$\epsilon > 0$. Since
\begin{equation} \label{eval_zeta2b}
\begin{split}
  &\zeta_2(2 \sigma_1, 2 \sigma_2) <
  \zeta_2^{\Box}(\sigma_1, \sigma_2) \\
  &= \sum_{k=2}^{\infty} k^{-2 \sigma_2} \Biggl(
    \sum_{\substack{m \mid k \\ m < \sqrt{k}}}
    \frac{1}{m^{\sigma_1 - \sigma_2}} \Biggr)^2 \\
  &\ll \sum_{k=2}^{\infty}
    \begin{cases}
      k^{-2 \sigma_2 + \epsilon} & (\sigma_1 \ge \sigma_2) \\
      k^{-\sigma_1 - \sigma_2 + \epsilon} & (\sigma_1 < \sigma_2)
    \end{cases}
\end{split}
\end{equation}
for any $\epsilon > 0$, the series $\zeta_2^{\Box}(\sigma_1, \sigma_2)$
converges if and only if $\sigma_2 > 1/2$ and $\sigma_1 + \sigma_2 > 1$.

\begin{lemma}[Theorem 5.2 in \cite{ivic}]  \label{lem_mth_dpoly}
Let $a_{1}, \cdots ,a_{N}$ be arbitrary complex numbers. Then
\begin{equation}  \label{eq_rama}
\int_{0}^{T} \biggl| \sum_{n\leq N}a_{n}n^{it} \biggr|^{2}dt=
T\sum_{n\leq N}|a_{n}|^{2}+O\left(\sum_{n\leq N}n|a_{n}|^{2}\right),
\end{equation}
and the above formula remains also valid if $N=\infty$,
provided that the series on the right-hand side of (\ref{eq_rama}) converges.
\end{lemma}

The following lemmas are well-known results for $\zeta(s)$
(see \cite{edwards} in p. 114 and Theorem 4.11 in \cite{titchmarsh}).

\begin{lemma}  \label{lem_em}
Let $s = \sigma + it \in \mathbb{C}$, $m,N \in \mathbb{N}$ and
$M=2m+1$. For $\sigma > -2m$ we have
\begin{align*}
  \zeta(s) = \sum_{n \le N} \frac{1}{n^s} &+ \frac{N^{1-s}}{s-1}
    - \frac{N^{-s}}{2} + \sum_{k=1}^{2m} \frac{B_{k+1}}{(k+1)!}
    (s)_k N^{-(s+k)} + \\
    &+ R_{M,N}(s),
\end{align*}
where
\[ R_{M,N}(s) = - \frac{(s)_M}{M!} \int_N^{\infty}
    B_M(x - [x]) x^{-s-M} dx. \]
\end{lemma}

\begin{corollary} \label{cor_em}
Let $s= 1+ it$. For fixed $t > 0$ we have
\[ \zeta(s) - \sum_{n \le N} \frac{1}{n^{s}} =
   \frac{N^{1-s}}{s-1} + O(N^{-1}) = O(1).\]
\end{corollary}

\begin{lemma}  \label{lem_hl}
Let $s = \sigma + it \in \mathbb{C}$.
We have
\[ \zeta(s) = \sum_{1 \le n \le x} \frac{1}{n^s} -
     \frac{x^{1-s}}{1-s} + O(x^{-\sigma}) \]
uniformly for $\sigma \ge \sigma_0 > 0$, $x \ge 1$,
$|t| \le 2 \pi x / C$, where $C$ is a given constant greater
than $1$.
\end{lemma}

The following lemma is an analogue of Lemma \ref{lem_hl}
for $\zeta^{\prime} (s)$.

\begin{lemma} \label{lem_zetapr}
Let $s = \sigma + it \in \mathbb{C}$.
We have
\[ \zeta'(s)=-\sum_{n\leq x}n^{-s}\log n-\frac{x^{-s+1}\log
x}{s-1}-\frac{x^{-s+1}}{(s-1)^{2}}+O(x^{-\sigma}\log x) \]
uniformly for $\sigma\geq\sigma_{0}>0$, $x\geq \exp(\sigma_{0}^{-1})$,
$|t|<2\pi x/C$, where $C$ is a given constant greater than $1$.
\end{lemma}

In order to prove this lemma we use the following well-known lemma
(see Lemma 4.10 in \cite{titchmarsh}).

\begin{lemma} \label{lem_exp_sum}
Let $f(x)$ be a real function with a continuous and steadily decreasing
derivative $f'(x)$ in $(a,b)$, and let $f'(b)=\alpha$, $f'(a)=\beta$ and
$g(x)$ be a real decreasing function, with a continuous derivative
$g'(x)$, and let $|g'(x)|$ be steadily decreasing.  Then
\begin{align*}
  \sum_{a<n\leq b}g(n)e^{2\pi if(n)}=
  &\sum_{\alpha-\eta<\nu<\beta+\eta} \int_{a}^{b}g(x)e^{2\pi
  i(f(x)-\nu x)}dx+ \\
    &+ O(g(a)\log(\beta-\alpha+2))+O(|g'(a)|),
\end{align*}
where $\eta$ is any positive constant less than $1$.
\end{lemma}

\begin{proof}[Proof of Lemma \ref{lem_zetapr}]
We have, by the Euler-Maclaurin formula,
\begin{align*}
  \zeta'(s)=- &\sum_{n\leq N}n^{-s}\log n-\int_{N}^{\infty}x^{-s}
     \log xdx-\frac{1}{2}N^{-s}\log N - \\
     &-\int_{N}^{\infty}\frac{d}{dx}
     \left(x^{-s}\log x\right)(x-[x]-1/2)dx.
\end{align*}
Since
\begin{align*}
  \int_{N}^{\infty}x^{-s}\log xdx&=\frac{1}{-s+1}\left([x^{-s+1}
  \log x]_{N}^{\infty}-\int_{N}^{\infty}x^{-s}dx\right)\\
  &=\frac{N^{-s+1}\log N}{s-1}+\frac{N^{-s+1}}{(s-1)^{2}},
\end{align*}
and
\begin{align*}
  &\int_{N}^{\infty}\frac{d}{dx}\left(x^{-s}\log x\right)(x-[x]-1/2)dx \\
  &=\int_{N}^{\infty}(-sx^{-s-1}\log x+x^{-s-1})(x-[x]-1/2)dx\\
  &=O(|s|N^{-\sigma}\log N),
\end{align*}
we have
\[ \zeta'(s)=-\sum_{n\leq N}n^{-s}\log n-\frac{N^{-s+1} \log N}{s-1}
     -\frac{N^{-s+1}}{(s-1)^{2}}+O(|s|N^{-\sigma}\log N). \]
The sum
\[ \sum_{x<n\leq N}n^{-\sigma}\log n \exp(it\log n) \]
is of the form considered in Lemma \ref{lem_exp_sum}, with
$g(u)=u^{-\sigma}\log u$, and
\[ f(u)=\frac{t\log u}{2\pi}. \]
Thus
\[ |f'(u)|\leq \frac{t}{2\pi x}<\frac{1}{C}. \]
Hence taking $\eta<1-C^{-1}$, we have
\[ \sum_{x<n\leq N}n^{-s}\log n=\int_{x}^{N}u^{-s}\log u du
     +O(x^{-\sigma}\log x). \]
Taking $N\rightarrow\infty$, the result follows.
\end{proof}

We use the following evaluations in this paper.

\begin{remark}   \label{rem_sumint}
Let $x_1, x_2$ be positive real variables with $1 \le x_1 \le x_2$.
For any fixed $\alpha, \beta > 0$,
\[ \sum_{x_1 \le r \le x_2} \frac{1}{r(r + \beta)^{\alpha}} \ll
   \frac{\log x_2}{(x_1 + \beta)^{\alpha}} \]
holds.
\end{remark}

\begin{remark}  \label{rem_sum_tlogt}
Let $T \ge 1$ and $M \ge 1$ with $M \ll \log T$.
For fixed $\alpha, \beta \ge 0$ we have
\begin{align*}
  \sum_{k \le M} \biggl( \frac{T}{2^k} \biggr)^{\alpha}
    \biggl( \log \biggl( \frac{T}{2^k} \biggr) \biggr)^{\beta}
  &\ll T^{\alpha} \sum_{k \le M} \biggl( \frac{1}{2^{\alpha}} \biggr)^k
    \bigl( (\log T)^{\beta} + k^{\beta} \bigr) \\
  &\ll \begin{cases}
         T^{\alpha}(\log T)^{\beta} & (\alpha \neq 0) \\
         (\log T)^{\beta + 1} & (\alpha = 0).
       \end{cases}
\end{align*}
\end{remark}

\section{Proof of Theorem \ref{th_intt1}}
In this section, we regard $s_2$ as a constant. We divide the
proof into two cases.

\begin{proof}[Proof of Theorem \ref{th_intt1} for
$\sigma_1 + \sigma_2 >2$]
We set
\[ a_m = \frac{1}{m^{\sigma_1}} \biggl( \zeta(s_2) -
     \sum_{n=1}^m \frac{1}{n^{s_2}} \biggr) \]
for $m \in \mathbb{N}$. If we assume $\sigma_2 > 1$ then we have
\begin{align*}
  \zeta_2(s_1, s_2) &= \sum_{m=1}^{\infty}
    \frac{1}{m^{\sigma_1 + it_1}}
    \sum_{n=m+1}^{\infty} \frac{1}{n^{s_2}} \\
  &= \sum_{m=1}^{\infty} a_m m^{-it_1}.
\end{align*}
The last series converges absolutely in $\sigma_1 + \sigma_2 > 2$.
Since
\[ \sum_{m=1}^{\infty} m |a_m|^2 = \sum_{m=1}^{\infty}
     \frac{1}{m^{2 \sigma_1 - 1}} \biggl| \zeta(s_2) -
     \sum_{n=1}^m \frac{1}{n^{s_2}} \biggr|^2 \]
converges by (\ref{eval_zeta21}), we have
\[ I^{[1]}(T) = \zeta_2^{[1]}(2 \sigma_1, s_2) T + O(1) \]
by Lemma \ref{lem_mth_dpoly}.
\end{proof}

In the case $3/2 \le \sigma_1 + \sigma_2 \le 2$, we use the
following lemma.

\begin{lemma}  \label{lem_approx_t1}
Let $s_1 = \sigma_1 + it_1, s_2 = \sigma_2 + it_2 \in \mathbb{C}$
with $t_1 \ge 1$ and $N \in \mathbb{N}$.
Let $C > 1$ be a given constant.
Assume that the point $(s_1, s_2) \in \mathbb{C}^2$ does not
encounter the singularities of $\zeta_2(s_1, s_2)$.
If $1 < |t_1 + t_2| < 2 \pi N/C$, then we have
\[ \zeta_2(s_1, s_2) = \sum_{m \le N} \frac{1}{m^{s_1}}
     \biggl( \zeta(s_2) - \sum_{n=1}^m \frac{1}{n^{s_2}} \biggr)
     + O(t_1^{-1} N^{2 - \sigma_1 - \sigma_2}) \]
for $\sigma_1 + \sigma_2 > 1$ and any fixed $s_2$.
\end{lemma}

\begin{proof}
Let $l \in \mathbb{N}$ with $\sigma_2 > -2l$.
First, we regard $s_1$ and $s_2$ as complex variables and
assume $\sigma_1, \sigma_2 > 1$. For any $N \in \mathbb{N}$,
we have
\[ \zeta_2(s_1, s_2) = \sum_{m=1}^N \frac{1}{m^{s_1}}
     \sum_{n=m+1}^{\infty} \frac{1}{n^{s_2}} +
     \sum_{m=N+1}^{\infty} \frac{1}{m^{s_1}}
     \sum_{n=m+1}^{\infty} \frac{1}{n^{s_2}} =
     V_1 + V_2, \]
say. Since
\[ V_1 = \sum_{m=1}^N \frac{1}{m^{s_1}}
     \biggl( \zeta(s_2) - \sum_{n=1}^m \frac{1}{n^{s_2}} \biggr), \]
$V_1$ is continued meromorphically to $\mathbb{C}^2$.
By setting $M=2l+1$ in Lemma \ref{lem_em}, we have
\begin{align*}
V_2 &= \sum_{m=N+1}^{\infty} \frac{1}{m^{s_1}} \biggl(
      \frac{m^{1-s_2}}{s_2 - 1} - \frac{m^{-s_2}}{2}
      + \sum_{k=1}^{M-1} \frac{B_{k+1}}{(k+1)!} (s_2)_k m^{-s_2 - k}
      + R_{M,m}(s_2) \biggr) \\
    &= \frac{1}{s_2 - 1} \sum_{m=N+1}^{\infty}
      \frac{1}{m^{s_1 + s_2 - 1}} - \frac{1}{2} \sum_{m=N+1}^{\infty}
      \frac{1}{m^{s_1 + s_2}} + \\
      &\quad + \sum_{k=1}^{M-1} \frac{B_{k+1}}{(k+1)!} (s_2)_k
      \sum_{m=N+1}^{\infty}
      \frac{1}{m^{s_1 + s_2 + k}} + \sum_{m=N+1}^{\infty}
      \frac{1}{m^{s_1}} R_{M,m}(s_2) \\
    &= \frac{1}{s_2 - 1} \biggl( \zeta(s_1 + s_2 - 1) -
      \sum_{m = 1}^N \frac{1}{m^{s_1 + s_2 - 1}} \biggr) -
      \frac{1}{2} \sum_{m=N+1}^{\infty} \frac{1}{m^{s_1 + s_2}} + \\
      &\quad + \sum_{k=1}^{M-1} \frac{B_{k+1}}{(k+1)!} (s_2)_k
      \sum_{m=N+1}^{\infty}
      \frac{1}{m^{s_1 + s_2 + k}} + \sum_{m=N+1}^{\infty}
      \frac{1}{m^{s_1}} R_{M,m}(s_2) \\
    &= I_1 + I_2 + I_3 + I_4,
\end{align*}
say. Since $I_4$ absolutely converges for
$\sigma_2 > -M + 1 = -2l$ and $\sigma_1 + \sigma_2 > -1$,
$V_2$ is continued meromorophically to
$\sigma_2 > -2l$ and $\sigma_1 + \sigma_2 > 1$.
Now, we regard $s_2$ as a constant.
By Lemma \ref{lem_hl}, we have
$I_1 \ll t_1^{-1} N^{2 - \sigma_1 - \sigma_2}$.
Also we can easily obtain
$I_2, I_3, I_4 \ll t_1^{-1} N^{2 - \sigma_1 - \sigma_2}$.
This implies the lemma.
\end{proof}

\begin{proof}[Proof of Theorem \ref{th_intt1} for
$3/2 \le \sigma_1 + \sigma_2 \le 2$]
Let
\[ a_{m}=m^{-\sigma_{1}}(\zeta(s_{2})-\sum_{n=1}^{m}n^{-s_{2}}) \]
and
\[ m_0 = \max \{ m \in \mathbb{N} \mid \frac{T}{2^m} > |t_2| + 1 \}. \]
Note that
\[ \sum_{m=1}^{\infty} |a_{m}|^2 = \zeta_2^{[1]}(2 \sigma_1, s_2) \]
in the case $\sigma_1 + \sigma_2 > 3/2$ and
\[ m_0 < \frac{\log T - \log (|t_2| + 1)}{\log 2} \le m_0 + 1 \]
hold.
We take $T \ge 2$ and $N \in \mathbb{N}$ with $|t_{2}|+1<T$
and $3T < 2 \pi N/C$, where $C > 1$, and we assume $T< t_1 < 2T$.
Then we have
\[ 1<t_{1}-|t_{2}|<|t_{1}+t_{2}|<|t_{1}|+|t_{2}|<3T<\frac{2\pi N}{C}. \]
Therefore we can use Lemma \ref{lem_approx_t1}, and we have
\[ \zeta_{2}(s_{1},s_{2})=\sum_{m=1}^{N}a_{m}m^{-it_{1}}
     + O(t_{1}^{-1}N^{2-\sigma_{1}-\sigma_{2}})
     = I_1 + I_2, \]
say. Since $a_{m}\ll m^{-\sigma_{1}-\sigma_{2}+1}$ by
Corollary \ref{cor_em}, we obtain
\[ \sum_{m=1}^{N}ma_{m}^{2} \ll \sum_{m=1}^{N}
     m^{3-2\sigma_{1}-2\sigma_{2}}\ll
     \begin{cases}
       \log N & (\sigma_{1}+\sigma_{2}=2) \\
       N^{4-2\sigma_{1}-2\sigma_{2}} & (\sigma_{1}+\sigma_{2}<2)
     \end{cases} \]
and
\[ I_{1} \ll \sum_{m=1}^{N}a_{m} \ll
     \sum_{m=1}^{N}m^{1-\sigma_{1}-\sigma_{2}}\ll
     \begin{cases}
       \log N & (\sigma_{1}+\sigma_{2}=2) \\
       N^{2-\sigma_{1}-\sigma_{2}} & (\sigma_{1}+\sigma_{2}<2).
     \end{cases} \]
Therefore we have
\[ \int_{T}^{2T}|I_{1}|^{2}dt_{1} = T\sum_{m=1}^{N}|a_{m}|^{2}+
   \begin{cases}
     O(\log N) & (\sigma_{1}+\sigma_{2}=2) \\
     O(N^{4-2\sigma_{1}-2\sigma_{2}}) & (\sigma_{1}+\sigma_{2}<2)
   \end{cases} \]
by Lemma \ref{lem_mth_dpoly} and
\[ \int_{T}^{2T}|I_{1}I_{2}|dt_{1} \ll
     N^{2-\sigma_{1}-\sigma_{2}}\max_{T<t_{1}<2T}|I_{1}| \ll
     \begin{cases}
       \log N & (\sigma_{1}+\sigma_{2}=2) \\
       N^{4-2\sigma_{1}-2\sigma_{2}} & (\sigma_{1}+\sigma_{2}<2).
     \end{cases} \]
On the other hand, we have
\[ \int_{T}^{2T}|I_{2}|^{2}dt_{1} \ll
     N^{4-2\sigma_{1}-2\sigma_{2}} \int_{T}^{2T}\frac{dt_{1}}{t_{1}^{2}}
     \ll T^{-1}N^{4-2\sigma_{1}-2\sigma_{2}}. \]
Therefore we have
\[ \int_{T}^{2T}|\zeta_{2}(s_{1},s_{2})|^{2}dt_{1} =
     T \sum_{m=1}^{N}|a_{m}|^{2}+
     \begin{cases}
       O(\log N) & (\sigma_{1}+\sigma_{2}=2) \\
       O(N^{4-2\sigma_{1}-2\sigma_{2}}) & (\sigma_{1}+\sigma_{2}<2).
     \end{cases} \]
By setting $N = [T] + 1$, we obtain
\begin{equation} \label{eval_intt1}
  \int_{T}^{2T}|\zeta_{2}(s_{1},s_{2})|^{2}dt_{1} =
    T \sum_{m \leq T}|a_{m}|^{2}+
     \begin{cases}
       O(\log T) & (\sigma_{1}+\sigma_{2}=2) \\
       O(T^{4-2\sigma_{1}-2\sigma_{2}}) & (\sigma_{1}+\sigma_{2}<2).
     \end{cases}
\end{equation}
Therefore, in the case $\sigma_1 + \sigma_2 > 3/2$, we have
\begin{equation*}
  \int_{T}^{2T}|\zeta_{2}(s_{1},s_{2})|^{2}dt_{1} =
    \zeta_2^{[1]} (2 \sigma_1, s_2) T +
      \begin{cases}
       O(\log T) & (\sigma_{1}+\sigma_{2}=2) \\
       O(T^{4-2\sigma_{1}-2\sigma_{2}}) & (3/2 < \sigma_{1}+\sigma_{2}<2).
     \end{cases}
\end{equation*}
By this relation and Remark \ref{rem_sum_tlogt}, we obtain
\begin{align*}
  &\int_{|t_2| + 1}^{T}|\zeta_{2}(s_{1},s_{2})|^{2}dt_{1} \\
  &= \int_{T/2^{m_0}}^{T}|\zeta_{2}(s_{1},s_{2})|^{2}dt_{1} + O(1) \\
  &= \sum_{1 \le k \le m_0} \int_{T/2^{k}}^{T/2^{k-1}}
    |\zeta_{2}(s_{1},s_{2})|^{2}dt_{1} + O(1) \\
  &= \zeta_2^{[1]} (2 \sigma_1, s_2) T \sum_{1 \le k \le m_0}
     \frac{1}{2^k} +
    \begin{cases}
      O \Bigl(\displaystyle{\sum_{1 \le k \le m_0}
        \log \frac{T}{2^k} \Bigr)} &
        (\sigma_{1}+\sigma_{2}=2) \\
      O \Bigl(\displaystyle{\sum_{1 \le k \le m_0}
        \Bigl( \frac{T}{2^k} \Bigr)^{4 - 2 \sigma_1 - 2 \sigma_2}
        \Bigr)} & (3/2 < \sigma_{1}+\sigma_{2}<2)
    \end{cases}  \\
  &= \zeta_2^{[1]} (2 \sigma_1, s_2) T +
       \begin{cases}
         O((\log T)^2) & (\sigma_{1}+\sigma_{2}=2) \\
         O(T^{4-2\sigma_{1}-2\sigma_{2}}) & (3/2 < \sigma_{1}+\sigma_{2}<2).
     \end{cases}
\end{align*}
This implies the theorem for $3/2 < \sigma_{1}+\sigma_{2} \le 2$.

In the case $\sigma_{1}+\sigma_{2} = 3/2$, since
\[ a_m = m^{-\sigma_{1}}
     \biggl(\zeta(s_{2})-\sum_{n=1}^{m}n^{-s_{2}}\biggr) =
   \frac{m^{1-\sigma_{1}-s_{2}}}{s_{2}-1}+O(m^{-\sigma_{1}-\sigma_{2}}) \]
by Lemma \ref{lem_em}, we have
\[ |a_{m}|^{2}=\frac{m^{-1}}{|s_{2}-1|^{2}}+O(m^{-2}). \]
Therefore we have
\[ \int_{T}^{2T}|\zeta_{2}(s_{1},s_{2})|^{2}dt_{1}
   = \frac{T\log T}{|s_{2}-1|^{2}}+O(T) \]
by (\ref{eval_intt1}).
By this relation and Remark \ref{rem_sum_tlogt}, we obtain
\begin{align*}
  &\int_{|t_2| + 1}^{T}|\zeta_{2}(s_{1},s_{2})|^{2}dt_{1} \\
  &= \sum_{1 \le k \le m_0} \int_{T/2^{k}}^{T/2^{k-1}}
    |\zeta_{2}(s_{1},s_{2})|^{2}dt_{1} + O(1) \\
  &= \frac{T}{|s_2 - 1|^2} \sum_{1 \le k \le m_0} \frac{1}{2^k}
    (\log T - k \log 2) + O(T) \\
  &= \frac{T\log T}{|s_{2}-1|^{2}} + \frac{T\log T}{|s_{2}-1|^{2}}
    \sum_{k > m_0} \frac{1}{2^k} - \frac{T \log 2}{|s_{2}-1|^{2}}
    \sum_{1 \le k \le m_0} \frac{k}{2^k} + O(T) \\
  &= \frac{T\log T}{|s_{2}-1|^{2}}+O(T).
\end{align*}
This completes the proof.
\end{proof}

\section{Proof of Theorem \ref{th_intt2}}
In this section, we regard $s_1$ as a constant. We divide the
proof into three cases.

\begin{proof}[Proof of Theorem \ref{th_intt2} for $\sigma_2 > 1$
and $\sigma_1 + \sigma_2 > 2$]
We set
\[ a_n = \frac{1}{n^{\sigma_2}} \sum_{m=1}^{n-1} \frac{1}{m^{s_1}} \]
for $n \in \mathbb{N}$. We have
\[ \zeta_2(s_1, s_2) = \sum_{n = 2}^{\infty} \biggl(
     \sum_{m=1}^{n-1} \frac{1}{m^{s_1}} \biggr)
     \frac{1}{n^{\sigma_2 + it_2}}
   = \sum_{n = 2}^{\infty} a_n n^{-it_2}. \]
Since
\[ \sum_{n=2}^{\infty} n |a_n|^2 = \sum_{n=2}^{\infty} \biggl(
     \sum_{m=1}^{n-1} \frac{1}{m^{s_1}} \biggr)
     \frac{1}{n^{2 \sigma_2 - 1}} \]
converges by (\ref{eval_zeta22}), we have
\[ I^{[2]}(T) = \zeta_2^{[2]}(s_1, 2 \sigma_2) T + O(1) \]
by Lemma \ref{lem_mth_dpoly}.
\end{proof}

We use the following lemma in the cases either $\sigma_1 > 1$,
$1/2 < \sigma_2 \le 1$ or $\sigma_1 \le 1$,
$3/2 \le \sigma_1 + \sigma_2 \le 2$.

\begin{lemma}  \label{lem_approx_t2}
Let $s_1 = \sigma_1 + it_1, s_2 = \sigma_2 + it_2 \in \mathbb{C}$
with $t_2 \ge 1$ and $N \in \mathbb{N}$ with $N > e^2$.
Let $C > 1$ be a given constant.
Assume that the point $(s_1, s_2) \in \mathbb{C}^2$ does not
encounter the singularities of $\zeta_2(s_1, s_2)$.
If $1 < t_2 < 2 \pi N/C$ and $1 < |t_1 + t_2| < 2 \pi N/C$, then
we have
\[ \zeta_2(s_1, s_2) = \sum_{2 \le n \le N} \biggl(
     \sum_{m=1}^{n-1} \frac{1}{m^{s_1}} \biggr)
     \frac{1}{n^{s_2}} +
   \begin{cases}
     O(t_2^{-1} N^{1 - \sigma_2}+ t_2^{-1} N^{2 - \sigma_1 - \sigma_2})
       & (s_1 \neq 1), \\
     O(t_2^{-1} N^{1 - \sigma_2} \log N) & (s_1 = 1)
   \end{cases} \]
for $\sigma_2 \ge 1/2$, $\sigma_1 + \sigma_2 > 1$ and any fixed $s_1$.
\end{lemma}

\begin{proof}
Let $l \in \mathbb{N}$ with $\sigma_1 > -2l$.
First, we regard $s_1$ and $s_2$ as complex variables and
assume $\sigma_1, \sigma_2 > 1$. For any $N \in \mathbb{N}$,
we have
\begin{equation} \label{eq_zeta2_u12}
  \zeta_2(s_1, s_2) = \sum_{2 \le n \le N} \biggl(
    \sum_{m=1}^{n-1} \frac{1}{m^{s_1}} \biggr)
    \frac{1}{n^{s_2}} + \sum_{n > N} \biggl(
    \sum_{m=1}^{n-1} \frac{1}{m^{s_1}} \biggr) = U_1 + U_2,
\end{equation}
say. The first term $U_1$ is obviously holomorphic in $\mathbb{C}^2$.
By setting $M=2l+1$ in Lemma \ref{lem_em}, we have
\begin{equation} \label{eq_zeta2_err}
\begin{split}
  U_2 &= \sum_{n > N} \biggl( \sum_{m=1}^{n} \frac{1}{m^{s_1}}
        - \frac{1}{n^{s_1}} \biggr) \frac{1}{n^{s_2}} \\
      &= \sum_{n > N} \biggl( \zeta(s_1) - \frac{n^{1-s_1}}{s_1 - 1}
        - \frac{n^{-s_1}}{2} - \sum_{k=1}^{M-1} \frac{B_{k+1}}{(k+1)!}
        (s_1)_k n^{-s_1 - k}
        - R_{M,n}(s_1) \biggr) \frac{1}{n^{s_2}} \\
      &= \zeta(s_1) \sum_{n > N} \frac{1}{n^{s_2}} + \frac{1}{1 - s_1}
        \sum_{n > N} \frac{1}{n^{s_1 + s_2 - 1}} - \frac{1}{2}
        \sum_{n > N} \frac{1}{n^{s_1 + s_2}} - \\
        &\quad - \sum_{k=1}^{M-1} \frac{B_{k+1}}{(k+1)!}
        (s_1)_k \sum_{n > N} \frac{1}{n^{s_1 + s_2 + k}}
        - \sum_{n > N} \frac{1}{n^{s_2}} R_{M,n}(s_1) \\
      &= \zeta(s_1) \biggl( \zeta(s_2) - \sum_{n=1}^N \frac{1}{n^{s_2}}
        \biggr) + \frac{1}{1 - s_1} \biggl( \zeta(s_1 + s_2 -1) -
        \sum_{n=1}^N \frac{1}{n^{s_1 + s_2 - 1}} \biggr) - \\
        &\quad - \frac{1}{2} \sum_{n > N} \frac{1}{n^{s_1 + s_2}} -
        \sum_{k=1}^{M-1} \frac{B_{k+1}}{(k+1)!} (s_1)_k
        \sum_{n > N} \frac{1}{n^{s_1 + s_2 + k}}
        - \sum_{n > N} \frac{1}{n^{s_2}} R_{M,n}(s_1) \\
      &= I_1 + I_2 + I_3 + I_4 + I_5,
\end{split}
\end{equation}
say. Since $I_5$ absolutely converges for
$\sigma_1 > -M + 1 = -2l$ and $\sigma_1 + \sigma_2 > -1$,
$U_2$ is continued meromorophically to $\sigma_2 > 0$,
$\sigma_1 > -2l$ and $\sigma_1 + \sigma_2 > 1$.
Now, we regard $s_1$ as a constant.

In the case $s_1 \neq 1$, by Lemma \ref{lem_hl}, we have
$I_1 \ll t_{2}^{-1}N^{1-\sigma_{2}}$ and
$I_2 \ll t_{2}^{-1}N^{2-\sigma_{1}-\sigma_{2}}$.
Also we can easily obtain
$I_3, I_4, I_5 \ll t_{2}^{-1}N^{2-\sigma_{1}-\sigma_{2}}$. This implies
the lemma for $s_1 \neq 1$.

Next we consider the case $s_1 = 1$. By (\ref{eq_zeta2_err}), we have
\[ U_2 = I_1 + I_2 + O(N^{1-\sigma_{1}-\sigma_{2}}). \]
Since we have
\[ I_2 = \frac{1}{1-s_1} \left(\zeta(s_{2}) -\sum_{n=1}^N
     \frac{1}{n^{s_2}}\right)
     - \left(\zeta'(s_{2}) +\sum_{n=1}^N \frac{\log n}{n^{s_2}}\right)
     + O(|s_1 - 1|) \]
by Taylor expansion, we obtain
\[ \zeta_{2}(1,s_{2})= U_1 + \gamma \biggl( \zeta(s_2) -
     \sum_{n=1}^N \frac{1}{n^{s_2}}\biggl) -
     \left(\zeta'(s_{2}) +\sum_{n=1}^N \frac{\log n}{n^{s_2}}\right)
     +O(N^{-\sigma_{2}}), \]
where $\gamma$ is Euler's constant.
By Lemma \ref{lem_hl} and Lemma \ref{lem_zetapr},
we have
\[ \zeta_{2}(1,s_{2})= U_1 + O(t_{2}^{-1}N^{1-\sigma_{2}}\log N). \]
This implies the lemma.
\end{proof}

\begin{proof}[Proof of Theorem \ref{th_intt2} for $s_1 \neq 1$]
We prove the theorem by the same argument as in the proof of
Theorem \ref{th_intt1}.

Let
\[ a_{n}=n^{-\sigma_{2}}\sum_{m=1}^{n-1}m^{-s_{1}}. \]
Note that
\[ \sum_{n=2}^{\infty} |a_{n}|^2 = \zeta_2^{[2]}(s_1, 2 \sigma_2) \]
in the case $\sigma_1 + \sigma_2 > 3/2$ and $\sigma_2 > 1/2$.
We take $T \ge 2$ and $N \in \mathbb{N}$ with $N > e^2$, $|t_{1}|+1<T$
and $3T < 2 \pi N/C$, where $C > 1$, and we assume $T< t_2 < 2T$.
Then we can use Lemma \ref{lem_approx_t2}, and we have
\[ \zeta_{2}(s_{1},s_{2})=\sum_{n=2}^{N}a_{n}n^{-it_{2}} +
     O(t_{2}^{-1}N^{1-\sigma_{2}} +
       t_{2}^{-1}N^{2-\sigma_{1}-\sigma_{2}})
   = I_1 + I_2, \]
say. Since
\[ a_{n} \ll
   \begin{cases}
     n^{-\sigma_{2}}& (\sigma_{1}\geq 1)\\
     n^{-\sigma_{1}-\sigma_{2}+1} & (\sigma_{1}<1)
   \end{cases} \]
by Corollary \ref{cor_em}, we obtain
\[ \sum_{n=2}^{N}na_{n}^{2}\ll
     \begin{cases}
       \log N & (\sigma_{2}=1,\sigma_{1}\geq1)\\
       N^{2-2\sigma_{2}} & (\sigma_{2}<1,\sigma_{1}\geq1) \\
       \log N & (\sigma_{1}+\sigma_{2}=2,\sigma_{1}<1)\\
       N^{4-2\sigma_{1}-2\sigma_{2}} &
         (\sigma_{1}+\sigma_{2}<2,\sigma_{1}<1)
     \end{cases} \]
and
\[ I_{1}\ll\sum_{n=2}^{N}a_{n}\ll
   \begin{cases}
     \log N & (\sigma_{2}=1,\sigma_{1}\geq1)\\
     N^{1-\sigma_{2}} & (\sigma_{2}<1,\sigma_{1}\geq1) \\
     \log N & (\sigma_{1}+\sigma_{2}=2,\sigma_{1}<1)\\
     N^{2-\sigma_{1}-\sigma_{2}}  &
       (\sigma_{1}+\sigma_{2}<2,\sigma_{1}<1).
   \end{cases} \]
Therefore we have
\[ \int_{T}^{2T}|I_{1}|^{2}dt_{2}=T\sum_{n=2}^{N}|a_{n}|^{2}+
     \begin{cases}
       O(\log N) & (\sigma_{2}=1,\sigma_{1}\geq1)\\
       O(N^{2-2\sigma_{2}}) & (\sigma_{2}<1,\sigma_{1}\geq1) \\
       O( \log N) & (\sigma_{1}+\sigma_{2}=2,\sigma_{1}<1)\\
       O(N^{4-2\sigma_{1}-2\sigma_{2}} ) &
         (\sigma_{1}+\sigma_{2}<2,\sigma_{1}<1)
   \end{cases} \]
by Lemma \ref{lem_mth_dpoly} and
\begin{align*}
\int_{T}^{2T}|I_{1}I_{2}|dt_{2} &\ll
   \begin{cases}
     (N^{1-\sigma_{2}}+N^{2-\sigma_{1}-\sigma_{2}})\log N&
       (\sigma_{2}=1,\sigma_{1}\geq1)\\
     (N^{1-\sigma_{2}}+N^{2-\sigma_{1}-\sigma_{2}})N^{1-\sigma_{2}} &
       (\sigma_{2}<1,\sigma_{1}\geq1) \\
     (N^{1-\sigma_{2}}+N^{2-\sigma_{1}-\sigma_{2}})\log N &
       (\sigma_{1}+\sigma_{2}=2,\sigma_{1}<1)\\
     (N^{1-\sigma_{2}}+N^{2-\sigma_{1}-\sigma_{2}})
       N^{2-\sigma_{1}-\sigma_{2}} &
       (\sigma_{1}+\sigma_{2}<2,\sigma_{1}<1)
   \end{cases} \\
&\ll \begin{cases}
       \log N & (\sigma_{2}=1,\sigma_{1}\geq1)\\
       N^{2-2\sigma_{2}} & (\sigma_{2}<1,\sigma_{1}\geq1) \\
       \log N & (\sigma_{1}+\sigma_{2}=2,\sigma_{1}<1)\\
       N^{4-2\sigma_{1}-2\sigma_{2}} &
         (\sigma_{1}+\sigma_{2}<2,\sigma_{1}<1).
     \end{cases}
\end{align*}
On the other hand, we have
\[ \int_{T}^{2T}|I_{2}|^{2}dt_{2} \ll
     T^{-1}(N^{2-2\sigma_{2}}+N^{4-2\sigma_{1}-2\sigma_{2}}). \]
Therefore we have
\[ \int_{T}^{2T}|\zeta_{2}(s_{1},s_{2})|^{2}dt_{2} =
     T\sum_{n=2}^{N}|a_{n}|^{2}+
     \begin{cases}
       O(\log N)& (\sigma_{2}=1,\sigma_{1}\geq1)\\
       O(N^{2-2\sigma_{2}}) &
         (\sigma_{2}<1,\sigma_{1}\geq1) \\
       O( \log N) & (\sigma_{1}+\sigma_{2}=2,\sigma_{1}<1)\\
       O(N^{4-2\sigma_{1}-2\sigma_{2}} ) &
         (\sigma_{1}+\sigma_{2}<2,\sigma_{1}<1).
   \end{cases} \]
By setting $N=[T]+1$, we obtain
\begin{equation} \label{eval_intt2}
\int_{T}^{2T}|\zeta_{2}(s_{1},s_{2})|^{2}dt_{2} =
  T\sum_{n \le T} |a_{n}|^{2}+
  \begin{cases}
       O(\log T)& (\sigma_{2}=1,\sigma_{1}\geq1)\\
       O(T^{2-2\sigma_{2}}) &
         (\sigma_{2}<1,\sigma_{1}\geq1) \\
       O( \log T) & (\sigma_{1}+\sigma_{2}=2,\sigma_{1}<1)\\
       O(T^{4-2\sigma_{1}-2\sigma_{2}} ) &
         (\sigma_{1}+\sigma_{2}<2,\sigma_{1}<1).
   \end{cases}
\end{equation}
Therefore, in the case $\sigma_1 + \sigma_2 > 3/2$ and
$\sigma_2 > 1/2$, we have
\begin{equation} \label{result_intt2_1}
\int_{T}^{2T}|\zeta_{2}(s_{1},s_{2})|^{2}dt_{2} =
  \zeta_2^{[2]}(s_1, 2 \sigma_2) T +
  \begin{cases}
       O(\log T)& (\sigma_{2}=1,\sigma_{1}\geq1)\\
       O(T^{2-2\sigma_{2}}) &
         (\sigma_{2}<1,\sigma_{1}\geq1) \\
       O( \log T) & (\sigma_{1}+\sigma_{2}=2,\sigma_{1}<1)\\
       O(T^{4-2\sigma_{1}-2\sigma_{2}} ) &
         (\sigma_{1}+\sigma_{2}<2,\sigma_{1}<1).
   \end{cases}
\end{equation}
In the case $\sigma_{1}>1, \sigma_{2}=1/2$, since
\[ a_{n}=n^{-\sigma_{2}}\sum_{m=1}^{n-1}m^{-s_{1}}=
     n^{-\sigma_{2}}(\zeta(s_{1})+O(n^{-\sigma_{1}+1})) \]
by Lemma \ref{lem_em}, we have
\[ |a_{n}|^{2}=n^{-1}|\zeta(s_{1})+O(n^{-\sigma_{1}+1})|^{2}=
     n^{-1}|\zeta(s_{1})|^{2}+O(n^{-\sigma_{1}}). \]
Therefore we obtain
\begin{equation} \label{result_intt2_2}
\int_{T}^{2T}|\zeta_{2}(s_{1},s_{2})|^{2}dt_{2}=
   |\zeta(s_{1})|^{2}T\log T+O(T)
\end{equation}
by (\ref{eval_intt2}).
In the case $\sigma_{1}<1$ and $\sigma_{1}+\sigma_{2}=3/2$, since
\[ a_n = n^{-\sigma_{2}} \Bigl( \frac{n^{-s_{1}+1}}{s_{1}-1}+
     O(n^{-\sigma_{1}})+O(1) \Bigr) \]
by Lemma \ref{lem_em}, we have
\begin{align*}
  |a_{n}|^{2} &= n^{-2\sigma_{2}}
    \biggl( \biggl|\frac{n^{-s_{1}+1}}{s_{1}-1} \biggr|^{2}
    +O(n^{-2\sigma_{1}+1})+O(n^{-\sigma_{1}+1}) \biggr) \\
  &= \frac{n^{-1}}{|{s_{1}-1}|^{2}}+O(n^{-2})+O(n^{-2+\sigma_{1}}).
\end{align*}
Therefore we obtain
\begin{equation} \label{result_intt2_3}
\int_{T}^{2T}|\zeta_{2}(s_{1},s_{2})|^{2}dt_{2}=
  \frac{T\log T}{|s_{1}-1|^{2}}+O(T)
\end{equation}
by (\ref{eval_intt2}).
In the case $\sigma_{1}=1$ and $\sigma_{2}=1/2$, since
\[ a_n = n^{-\sigma_{2}} \biggl( \zeta(s_{1})-
     \frac{n^{-s_{1}+1}}{s_{1}-1}+O(n^{-\sigma_{1}}) \biggr) \]
by Lemma \ref{lem_em}, we have
\[ \sum_{n \le T}|a_{n}|^{2}=\sum_{n \le T} \biggl(n^{-1} \biggl|\zeta(s_{1})-
     \frac{n^{-s_{1}+1}}{s_{1}-1} \biggr|^{2}+
     O(n^{-2}) \biggr) \]
by Corollary \ref{cor_em}. Since we have
\begin{align*}
&\sum_{n \le T}n^{-1} \biggl|\zeta(s_{1})-
  \frac{n^{-s_{1}+1}}{s_{1}-1} \biggr|^{2} \\
&= (|\zeta(s_{1})|^{2} +|s_{1}-1|^{-2})\log T
  -2 \sum_{n \le T}\Re\left(\overline{\zeta(s_{1})}
  \frac{n^{-s_{1}}}{s_{1}-1}\right) + O(1) \\
&= (|\zeta(s_{1})|^{2} +|s_{1}-1|^{-2})\log T + O(1)
\end{align*}
by Corollary \ref{cor_em}, we have
\[ \sum_{n \le T}|a_{n}|^{2}=(|\zeta(s_{1})|^{2}+|s_{1}-1|^{-2})\log T
     + O(1). \]
Therefore we obtain
\begin{equation} \label{result_intt2_4}
\int_{T}^{2T}|\zeta_{2}(s_{1},s_{2})|^{2}dt_{2}=
(|\zeta(s_{1})|^{2}+|s_{1}-1|^{-2})T\log T+O(T)
\end{equation}
by (\ref{eval_intt2}). By (\ref{result_intt2_1}), (\ref{result_intt2_2}),
(\ref{result_intt2_3}) and (\ref{result_intt2_4}), we can obtain
the theorem by the same argument as in the proof of Theorem
\ref{th_intt1}.
\end{proof}

\begin{proof}[Proof of Theorem \ref{th_intt2} for $s_1 = 1$]
We prove the theorem by the same argument as in the proof of
Theorem \ref{th_intt1}.

Hereafter we use the same notations as in the previous proof.
Note that, in this case, we have
$I_2 = O(t_{2}^{-1}N^{1-\sigma_{2}}\log N)$
by using Lemma \ref{lem_approx_t2}.
Since $a_{n}\ll n^{-\sigma_{2}}\log n$, we obtain
\[ \sum_{n=2}^{N}n|a_{n}|^{2}\ll\sum_{n=2}^{N}
     n^{1-2\sigma_{2}}(\log n)^2 \ll
     \begin{cases}
       O((\log N)^3) & (\sigma_{2}=1)\\
       O(N^{2-2\sigma_{2}}(\log N)^2) & (\sigma_{2}<1)\\
        \end{cases} \]
and
\[ I_{1}\ll\sum_{n=2}^{N}|a_{n}|\ll
     \begin{cases}
       (\log N)^2 & (\sigma_{2}=1)\\
       N^{1-\sigma_{2}}\log N & (\sigma_{2}<1). \\
   \end{cases} \]
Therefore we have
\[ \int_{T}^{2T}|I_{1}|^{2}dt_{2}=T\sum_{n=2}^{N}|a_{n}|^{2}+
     \begin{cases}
       O((\log N)^3) & (\sigma_{2}=1)\\
       O(N^{2-2\sigma_{2}} (\log N)^2) & (\sigma_{2}<1)\\
   \end{cases} \]
by Lemma \ref{lem_mth_dpoly} and
\[ \int_{T}^{2T}|I_{1}I_{2}|dt_{2}\ll
     \begin{cases}
       O((\log N)^3) & (\sigma_{2}=1)\\
       O(N^{2-2\sigma_{2}} (\log N)^2) & (\sigma_{2}<1).\\
\end{cases} \]
On the other hand, we have
\[ \int_{T}^{2T}|I_{2}|^{2}dt_{2}\ll T^{-1}N^{2-2\sigma_{2}}(\log N)^2. \]
Therefore, by setting $N = [T] + 1$, we obtain
\begin{equation} \label{eval_intt2_s11}
\int_{T}^{2T}|\zeta_{2}(s_{1},s_{2})|^{2}dt_{2}=
   T\sum_{n \le T}|a_{n}|^{2}+
   \begin{cases}
     O((\log T)^3)& (\sigma_{2}=1)\\
     O(T^{2-2\sigma_{2}} (\log T)^2) & (\sigma_{2}<1). \\
   \end{cases}
\end{equation}
In the case $\sigma_{2}>1/2$, we have
\begin{equation} \label{result_intt2_s11_1}
\int_{T}^{2T}|\zeta_{2}(s_{1},s_{2})|^{2}dt_{2}=
  \zeta_2^{[2]}(s_1, 2 \sigma_2) T+
  \begin{cases}
    O((\log T)^3)& (\sigma_{2}=1)\\
    O(T^{2-2\sigma_{2}} (\log T)^2) & (\sigma_{2}<1).
   \end{cases}
\end{equation}
In the case $\sigma_{2}=1/2$, since
\[ |a_{n}|^{2}=n^{-1} \Bigl(\sum_{m=1}^{n-1}m^{-1} \Bigr)^{2}=
     \frac{(\log n)^2}{n}+O \Bigl( \frac{\log n}{n} \Bigr) \]
and
\[ \sum_{n=2}^N \frac{(\log n)^2}{n}
   = \int_{1}^{N}x^{-1}(\log x)^2 dx + O(1)
   = \frac{(\log N)^3}{3} + O(1)\]
hold, we have
\begin{equation}  \label{result_intt2_s11_2}
\int_{T}^{2T}|\zeta_{2}(s_{1},s_{2})|^{2}dt_{2}=
  \frac{T(\log T)^3}{3}+O(T(\log T)^2)
\end{equation}
by (\ref{eval_intt2_s11}).
By (\ref{result_intt2_s11_1}) and (\ref{result_intt2_s11_2}),
we can obtain the theorem by the same argument as in the proof
of Theorem \ref{th_intt1}.
\end{proof}

\section{Proof of Theorem \ref{th_intt}}
We divide the proof into four cases.

\begin{proof}[Proof of Theorem \ref{th_intt} for $\sigma_2 > 1$
and $\sigma_1 + \sigma_2 > 2$]
We set
\[ a_k = \biggl( \sum_{\substack{m \mid k \\ m < \sqrt{k}}}
           \frac{1}{m^{\sigma_1 - \sigma_2}} \biggr)
           \frac{1}{k^{\sigma_2}} \]
for $k \in \mathbb{N}$. We have
\begin{align*}
  \zeta_2(s_1, s_2) &= \sum_{1 \le m < n}
    \frac{1}{m^{\sigma_1} n^{\sigma_2} (mn)^{it}} \\
  &= \sum_{k \ge 2} \biggl( \sum_{\substack{mn = k \\ m < n}}
    \frac{1}{m^{\sigma_1} n^{\sigma_2}} \biggr)
    \frac{1}{k^{it}} \\
  &= \sum_{k \ge 2} \biggl( \sum_{\substack{m \mid k \\ m < \sqrt{k}}}
    \frac{1}{m^{\sigma_1 - \sigma_2}} \biggr)
    \frac{1}{k^{\sigma_2 + it}} \\
  &= \sum_{k \ge 2} a_k k^{-it}.
\end{align*}
Since
\[ \sum_{k \ge 2} k |a_k|^2 = \sum_{k \ge 2}^{\infty} 
     \biggl( \sum_{\substack{m \mid k \\ m < \sqrt{k}}}
     \frac{1}{m^{\sigma_1 - \sigma_2}} \biggr)^2
     \frac{1}{k^{2 \sigma_2 - 1}} \]
converges by (\ref{eval_zeta2b}), we have
\[ I^{\Box}(T) = \zeta_2^{\Box} (\sigma_1, \sigma_2) T + O(1) \]
by Lemma \ref{lem_mth_dpoly}.
\end{proof}

We use the following lemma in the cases either $\sigma_1 > 1$,
$1/2 < \sigma_2 \le 1$ or $\sigma_1 \le 1$,
$3/2 < \sigma_1 + \sigma_2 \le 2$.

\begin{lemma}  \label{lem_approx}
Let $\sigma_1 + \sigma_2 > 1$, $\sigma_2 > 0$,
$s_1 = \sigma_1 + it$ and $s_2 = \sigma_2 + it$. Then
\[ \zeta_2(s_1 , s_2) = \sum_{n \leq t} n^{-s_2} \sum^{n - 1}_{m = 1} m^{-s_1}
   + \begin{cases}
       O(t^{-\sigma_2}) & (\sigma_1 > 1) \\
       O(t^{-\sigma_2 + \epsilon}) & (\sigma_1 = 1) \\
       O(t^{1 - \sigma_1 - \sigma_2}) & (\sigma_1 < 1)
     \end{cases} \]
holds for $t \ge 2$.
\end{lemma}

In order to prove Lemma \ref{lem_approx}, we use the following
lemma and corollary.

\begin{lemma}[Lemma 2.2 in \cite{kiu_tani}] \label{lem_kiu_tani}
Let $s = \sigma + it$, $|t| > 1$. For $N > \frac{1}{4} |t|$,
$m \ge 1$ and $\sigma > -2m -1$, we have
\begin{align*}
  \zeta(s) = \sum_{n \le N} \frac{1}{n^s} &+ \frac{N^{1-s}}{s-1}
     - \frac{N^{-s}}{2} + \sum_{k=1}^{2m} \frac{B_{k+1}}{(k+1)!}
     (s)_k N^{-(s+k)} + \\
     &+ O \bigl( |t|^{2m+1} N^{- \sigma - 2m - 1} \bigr),
\end{align*}
where the implied constant does not depend on $t$.
\end{lemma}

\begin{corollary}[Corollary 2.3 in \cite{kiu_tani}] \label{cor_kiu_tani}
Let $s = \sigma + it$, $|t| > 1$. For $N > \frac{1}{4} |t|$ and
$\sigma > -3$, we have
\[ \zeta(s) = \sum_{n \le N} \frac{1}{n^s} + \frac{N^{1-s}}{s-1}
     - \frac{N^{-s}}{2} + \frac{s}{12} N^{-s-1} +
     O \bigl( |t|^3 N^{- \sigma - 3} \bigr), \]
where the implied constant does not depend on $t$.
\end{corollary}

The following proof is similar to that in \cite{kiu_tani}
(section 4.1 Evaluation of $S_2(s_1, s_2)$).

\begin{proof}[Proof of Lemma \ref{lem_approx}]
Let $l \in \mathbb{N}$ with $\sigma_1 > -2l$.
We use (\ref{eq_zeta2_u12}) and (\ref{eq_zeta2_err}).
Hence we obtain the analytic continuation of $\zeta_2(s_1,s_2)$
for $\sigma_2 > 0$, $\sigma_1 > -2l$ and $\sigma_1 + \sigma_2 > 1$.
Now, we set $s_1 = \sigma_1 + it$, $s_2 = \sigma_2 + it$
with $t \ge 1$ and $N = [t]$. Then we have
\begin{align*}
  I_1 &= \zeta(s_1) \Biggl( \frac{N^{1-s_2}}{s_2 - 1} - \frac{N^{-s_2}}{2} +
        \frac{s_2}{12} N^{-s_2 - 1} + O \bigl(
        |t|^3 N^{-\sigma_2 - 3} \bigl) \Biggr) \\
      &\ll | \zeta(s_1) | t^{-\sigma_2} \\
      &\ll
      \begin{cases}
        t^{-\sigma_2} & (\sigma_1 > 1) \\
        t^{-\sigma_2 + \epsilon} & (\sigma_1 = 1) \\
        t^{1-\sigma_1 -\sigma_2} & (\sigma_1 < 1)
      \end{cases}
\end{align*}
for $\sigma_2 > -3$ by Corollary \ref{cor_kiu_tani}.
Similarly, we have
\begin{align*}
 I_2 &= \frac{1}{1-s_1} \Biggl(
       \frac{N^{2-s_1 -s_2}}{s_1 + s_2 -2} - \frac{1}{2}
       N^{1 -s_1 -s_2} + \frac{s_1 + s_2 -1}{12} N^{-s_1 - s_2} + \\
       &\qquad + O \bigl( |t|^3 N^{1-\sigma_1 -\sigma_2 -3} \bigr)
       \Biggr) \\
     &\ll t^{1-\sigma_1 -\sigma_2}
\end{align*}
for $\sigma_1 + \sigma_2 > -2$.
Since $\sigma_1 + \sigma_2 > 1$, we have
\[ I_j \ll t^{1-\sigma_1 -\sigma_2} \qquad (j=3,4). \]
On the other hand, $R_{M,n}(s_1) = O(t^M n^{- \sigma_1 -M})$
for $\sigma_1 > -M$ by Lemma \ref{lem_kiu_tani}. Hence we have
\[ I_5 \ll t^M \sum_{n > N} \frac{1}{n^{\sigma_1 + \sigma_2 + M}}
     \ll t^{1-\sigma_1 -\sigma_2}. \]
This implies the lemma.
\end{proof}

First, we consider the case $\sigma_1 > \sigma_2$. Especially,
this condition is satisfied when $\sigma_1 > 1$ and $1/2 < \sigma_2 \le 1$.

\begin{proof}[Proof of Theorem \ref{th_intt}
for $\sigma_1 > \sigma_2$]
If we set
\begin{align*}
  A(s_1 , s_2) = \sum_{n \leq t} n^{-s_2} \sum^{n - 1}_{m = 1} m^{-s_1}
\end{align*}
then we have
\begin{align*}
  &\int^T_2 |A(s_1 , s_2)|^2 dt =
    \int^T_2 \biggl( \sum_{n_1 \leq t} n_1^{-s_2}
    \sum^{n_1 - 1}_{m_1 = 1} m_1^{-s_1}
    \sum_{n_2 \leq t} n_2^{-\overline{s_2}}
    \sum^{n_2 - 1}_{m_2 = 1} m_2^{-\overline{s_1}} \biggr) dt \\
  &= \sum_{2 \leq n_1 \leq T} \sum^{n_1 - 1}_{m_1 = 1}
    \sum_{2 \leq n_2 \leq T} \sum^{n_2 - 1}_{m_2 = 1}
    n_1^{-\sigma_2} m_1^{-\sigma_1} n_2^{-\sigma_2} m_2^{-\sigma_1}
    \int^T_{M(n_1 , n_2)} \biggl( \frac{m_2 n_2}{m_1 n_1} \biggr)^{it} dt \\
  &= \sum_{m_1 n_1 = m_2 n_2}
    \sum_{\substack{1 \leq m_1 \leq n_1 - 1 \\ 2 \leq n_1 \leq T}}
    \sum_{\substack{1 \leq m_2 \leq n_2 - 1 \\ 2 \leq n_2 \leq T}}
    n_1^{-\sigma_2} m_1^{-\sigma_1} n_2^{-\sigma_2} m_2^{-\sigma_1}
    \biggl(T - M(n_1 , n_2) \biggr) \\
  &\ \ \ + \sum_{m_1 n_1 \neq m_2  n_2}
    \sum_{\substack{1 \leq m_1 \leq n_1 - 1 \\ 2 \leq n_1 \leq T}}
    \sum_{\substack{1 \leq m_2 \leq n_2 - 1 \\ 2 \leq n_2 \leq T}}
    n_1^{-\sigma_2} m_1^{-\sigma_1} n_2^{-\sigma_2} m_2^{-\sigma_1} \\
  & \ \ \ \ \ \ \ \times
    \frac{\exp \biggl(iT \log \Bigl( \frac{m_2 n_2}{m_1 n_1} \Bigr) \biggr)
    - \exp \biggl( i M(n_1 , n_2)
    \log \Bigl( \frac{m_2 n_2}{m_1 n_1} \Bigr) \biggr)}
    {i \log \Bigl(\frac{m_2 n_2}{m_1 n_1} \Bigr)} \\
  &= S_1 T - S_2 + S_3,
\end{align*}
say, where $M(n_1 , n_2) = \max(n_1 , n_2)$. First, we rewrite
\begin{align*}
  S_1 &= \sum_{2 \leq k \leq T}
    \biggl( \sum_{\substack{mn = k \\ m < n \le T}}
    m^{-\sigma_1} n^{-\sigma_2} \biggr)^2 +
    \sum_{T < k < T^2} \biggl( \sum_{\substack{mn = k \\ m < n \le T}}
    m^{-\sigma_1} n^{-\sigma_2} \biggr)^2 \\
  &= \sum^{\infty}_{k = 2}
    \biggl( \sum_{\substack{mn = k \\ m < n}}
    m^{-\sigma_1} n^{-\sigma_2} \biggr)^2
    - \sum_{k > T} \biggl( \sum_{\substack{mn = k \\ m < n}}
    m^{-\sigma_1} n^{-\sigma_2} \biggr)^2
    + \sum_{T < k < T^2} \biggl( \sum_{\substack{mn = k \\ m < n \leq T}}
    m^{-\sigma_1} n^{-\sigma_2} \biggr)^2.
\end{align*}
Since
\begin{align*}
  \sum_{k > T} \biggl( \sum_{mn = k} m^{-\sigma_1} n^{-\sigma_2} \biggr)^2
  &= \sum_{k > T} \biggl( \sum_{m|k}
    m^{-\sigma_1} m^{\sigma_2} k^{-\sigma_2} \biggr)^2 \\
  &\ll \sum_{k > T} k^{-2 \sigma_2 + \epsilon}
    \ll T^{1 - 2 \sigma_2 + \epsilon},
\end{align*}
we have
\[  S_1 = \zeta_2^{\Box}(\sigma_1, \sigma_2)
      + O(T^{1 - 2 \sigma_2 + \epsilon}). \]
Next, we rewrite
\begin{align*}
  S_2 &= \sum_{m_1 n_1 = m_2 n_2}
    \sum_{\substack{1 \leq m_1 \leq n_1 - 1 \\ 2 \leq n_1 \leq T}}
    \sum_{\substack{1 \leq m_2 \leq n_2 - 1 \\ 2 \leq n_2 \leq T}}
    n_1^{-\sigma_2} m_1^{-\sigma_1}
    n_2^{-\sigma_2} m_2^{-\sigma_1} M(n_1 , n_2) \\
  &= \sum_{2 \leq k \leq T}
    \sum_{\substack{m_1 n_1 = k \\ 1 \leq m_1 < n_1}}
    m_1^{-\sigma_1} n_1^{-\sigma_2}
    \biggl( \sum_{\substack{m_2 n_2 = k \\ 1 \leq m_2 < n_2 \\ n_1 < n_2}}
    n_2^{1 - \sigma_2} m_2^{-\sigma_1} +
    \sum_{\substack{m_2 n_2 = k \\ 1 \leq m_2 < n_2 \\ n_1 \geq n_2}}
    n_1 n_2^{-\sigma_2} m_2^{-\sigma_1} \biggr) \\
    & \ \ \ + \sum_{T < k < T^2}
    \sum_{\substack{m_1 n_1 = k \\ m_1 < n_1 \leq T}}
    m_1^{-\sigma_1} n_1^{-\sigma_2}
    \biggl( \sum_{\substack{m_2 n_2 = k \\ m_2 < n_2 \leq T \\ n_1 < n_2}}
    n_2^{1 - \sigma_2} m_2^{-\sigma_1} +
    \sum_{\substack{m_2 n_2 = k \\ m_2 < n_2 \leq T \\ n_1 \geq n_2}}
    n_1 n_2^{-\sigma_2} m_2^{-\sigma_1} \biggr) \\
  &= A_1 + A_2 + A_3 + A_4,
\end{align*}
say. Since
\begin{align*}
  A_1 &= \sum_{2 \leq k \leq T}
    \sum_{\substack{m_1 n_1 = k \\ 1 \leq m_1 < n_1}}
    m_1^{-\sigma_1} n_1^{-\sigma_2}
    \biggl(\sum_{\substack{m_2 n_2 = k \\ 1 \leq m_2 < n_2 \\ n_1 < n_2}}
    n_2^{1 - \sigma_2} m_2^{-\sigma_1} \biggr) \\
  &= \sum_{2 \leq k \leq T}
    \sum_{\substack{m_1 n_1 = k \\ 1 \leq m_1 < n_1}}
    m_1^{-\sigma_1} n_1^{-\sigma_2}
    \sum_{\substack{m_2 | k \\ m_2 < \sqrt{k} \\ m_2 < \frac{k}{n_1}}}
    k^{1 - \sigma_2} m_2^{\sigma_2 - 1 - \sigma_1} \\
  &\ll \sum_{2 \leq k \leq T} k^{1 - \sigma_2 + \epsilon}
    \sum_{\substack{m_1 | k \\ m_1 < \sqrt{k}}}
    m_1^{\sigma_2 - \sigma_1} k^{-\sigma_2} \\
  &\ll \sum_{2 \leq k \leq T} k^{1 - 2 \sigma_2 + \epsilon}
    \ll T^{2 - 2 \sigma_2 + \epsilon},
\end{align*}

\begin{align*}
  A_2 &= \sum_{2 \leq k \leq T}
    \sum_{\substack{m_1 n_1 = k \\ 1 \leq m_1 < n_1}}
    m_1^{-\sigma_1} n_1^{1 - \sigma_2}
    \sum_{\substack{m_2 n_2 = k \\ 1 \leq m_2 < n_2 \\ n_1 \geq n_2}}
    n_2^{- \sigma_2} m_2^{-\sigma_1} \\
  &= \sum_{2 \leq k \leq T}
    \sum_{\substack{m_1 n_1 = k \\ 1 \leq m_1 < n_1}}
    m_1^{-\sigma_1} n_1^{-\sigma_2 + 1}
    \sum_{\substack{m_2 | k \\ 1 \leq m_2 < \sqrt{k} \\ \frac{k}{m_2} \le n_1}}
    k^{- \sigma_2} m_2^{\sigma_2 - \sigma_1} \\
  &\ll \sum_{2 \leq k \leq T}
    \sum_{\substack{m_1 n_1 = k \\ 1 \leq m_1 < n_1}}
    m_1^{-\sigma_1} n_1^{-\sigma_2 + 1} k^{-\sigma_2 + \epsilon} \\
  &= \sum_{2 \leq k \leq T} \sum_{\substack{m_1 | k \\ m_1 < \sqrt{k}}}
    m_1^{-\sigma_1 + \sigma_2 - 1} k^{-2 \sigma_2 + 1 + \epsilon} \\
  &\ll \sum_{2 \leq k \leq T} k^{-2 \sigma_2 + 1 + \epsilon}
    \ll T^{2 - 2 \sigma_2 + \epsilon},
\end{align*}

\begin{align*}
  A_3 &= \sum_{T < k < T^2}
    \sum_{\substack{m_1 n_1 = k \\ m_1 < n_1 \leq T}}
    m_1^{-\sigma_1} n_1^{-\sigma_2}
    \sum_{\substack{m_2 n_2 = k \\ m_2 < n_2 \leq T \\ n_1 < n_2}}
    n_2^{1 - \sigma_2} m_2^{-\sigma_1} \\
  &= \sum_{T < k < T^2}
    \sum_{\substack{m_1 n_1 = k \\ m_1 < n_1 \leq T}}
    m_1^{-\sigma_1} n_1^{-\sigma_2}
    \sum_{\substack{m_2 | k \\ m_2 < \sqrt{k} \\ n_1 < \frac{k}{m_2} \leq T}}
    k^{1 - \sigma_2} m_2^{-1 + \sigma_2 - \sigma_1} \\
  &\ll \sum_{T < k < T^2}
    \sum_{\substack{m_1 n_1 = k \\ m_1 < n_1 \leq T}}m_1^{-\sigma_1}
    n_1^{-\sigma_2} k^{1 - \sigma_2}
    \biggl( \frac{k}{T} \biggr)^{-1 + \sigma_2 - \sigma_1}
    k^{\epsilon} \\
  &= T^{1 - \sigma_2 + \sigma_1}
    \sum_{T < k < T^2} k^{-\sigma_1 + \epsilon}
    \sum_{\substack{m_1 | k \\ m_1 < \sqrt{k} \\ \frac{k}{T} \leq m_1}}
    m_1^{-\sigma_1} k^{-\sigma_2} m_1^{\sigma_2} \\
  &\ll T^{1 - \sigma_2 + \sigma_1} \sum_{T < k < T^2}
    k^{-\sigma_1 - \sigma_2 + \epsilon}
    \biggl( \frac{k}{T} \biggr)^{-\sigma_1 + \sigma_2} \\
  &= T^{1 - 2 \sigma_2 + 2 \sigma_1} \sum_{T < k < T^2}
    k^{-2 \sigma_1 + \epsilon}
    \ll T^{2 - 2 \sigma_2 + \epsilon}
\end{align*}
and
\begin{align*}
  A_4 &= \sum_{T < k < T^2}
    \sum_{\substack{m_1 n_1 = k \\ m_1 < n_1 \leq T}}
    m_1^{-\sigma_1} n_1^{1 - \sigma_2}
    \sum_{\substack{m_2 n_2 = k \\ m_2 < n_2 \leq T \\ n_2 \leq n_1}}
    n_2^{-\sigma_2} m_2^{-\sigma_1} \\
  &= \sum_{T < k < T^2}
    \sum_{\substack{m_1 n_1 = k \\ m_1 < n_1 \leq T}}
    m_1^{-\sigma_1} n_1^{1 - \sigma_2}
    \sum_{\substack{m_2 | k \\ \frac{k}{T} \leq m_2 < \sqrt{k} \\
    \frac{k}{n_1} \leq m_2}}
    k^{-\sigma_2} m_2^{\sigma_2 - \sigma_1} \\
  &\ll \sum_{T < k < T^2}
    \sum_{\substack{m_1 n_1 = k \\ m_1 < n_1 \leq T}}
    m_1^{-\sigma_1} n_1^{1 - \sigma_2} k^{-\sigma_2}
    \biggl( \frac{k}{n_1} \biggr)^{\sigma_2 - \sigma_1} k^{\epsilon} \\
  &= \sum_{T < k < T^2}
    \sum_{\substack{m_1 n_1 = k \\ m_1 < n_1 \leq T}}
    m_1^{-\sigma_1} n_1^{1 + \sigma_1 - 2 \sigma_2}
    k^{-\sigma_1 + \epsilon} \\
  &= \sum_{T < k < T^2}
    \sum_{\substack{m_1 | k \\ m_1 < \sqrt{k} \\ \frac{k}{T} \leq m_1}}
    m_1^{-\sigma_1} \biggl( \frac{k}{m_1} \biggr)^{1 + \sigma_1 - 2 \sigma_2}
    k^{-\sigma_1 + \epsilon} \\
  &= \sum_{T < k < T^2}
    \sum_{\substack{m_1 | k \\ m_1 < \sqrt{k} \\ \frac{k}{T} \leq m_1}}
    k^{1 - 2 \sigma_2 + \epsilon} m_1^{-1 - 2 \sigma_1 + 2 \sigma_2} \\
  &\ll \sum_{T < k < T^2} k^{1 - 2 \sigma_2 + \epsilon}
    \biggl( \frac{k}{T} \biggr)^{-1 - 2 \sigma_1 + 2 \sigma_2} \\
  &= T^{1 + 2 \sigma_1 - 2 \sigma_2} \sum_{T < k < T^2}
    k^{-2 \sigma_1 + \epsilon}
    \ll T^{2 - 2 \sigma_2 + \epsilon},
\end{align*}
we have $S_2 \ll T^{2 - 2 \sigma_2 + \epsilon}$.
Next, we have
\begin{align*}
  S_{3} &=\sum_{m_1 n_1 \neq m_2 n_2}
    \sum_{\substack{1 \leq m_1 \leq n_1 - 1 \\ 2 \leq n_1 \leq T}}
    \sum_{\substack{1 \leq m_2 \leq n_2 - 1 \\ 2 \leq n_2 \leq T}}
    n_1^{-\sigma_2} m_1^{-\sigma_1}
    n_2^{-\sigma_2} m_2^{-\sigma_1} \\
  & \ \ \ \ \ \ \ \times
    \frac{\exp \biggl(iT \log \Bigl( \frac{m_2 n_2}{m_1 n_1} \Bigr) \biggr)
    - \exp \biggl( i M(n_1 , n_2)
    \log \Bigl( \frac{m_2 n_2}{m_1 n_1} \Bigr) \biggr)}
    {i \log \Bigl(\frac{m_2 n_2}{m_1 n_1} \Bigr)} \\
  &\ll \sum_{m_1 n_1 < m_2 n_2}
    \sum_{\substack{1 \leq m_1 \leq n_1 - 1 \\ 2 \leq n_1 \leq T}}
    \sum_{\substack{1 \leq m_2 \leq n_2 - 1 \\ 2 \leq n_2 \leq T}}
    n_1^{-\sigma_2} m_1^{-\sigma_1}n_2^{-\sigma_2} m_2^{-\sigma_1}
    \frac{1}{\log \left(\frac{m_{2}n_{2}}{m_{1}n_{1}}\right)} \\
  &= \sum_{m_1 n_1 < m_2 n_2<2m_{1}n_{1}}
    \sum_{\substack{1 \leq m_1 \leq n_1 - 1 \\ 2 \leq n_1 \leq T}}
    \sum_{\substack{1 \leq m_2 \leq n_2 - 1 \\ 2 \leq n_2 \leq T}}
    n_1^{-\sigma_2} m_1^{-\sigma_1}n_2^{-\sigma_2} m_2^{-\sigma_1}
    \frac{1}{\log \left(\frac{m_{2}n_{2}}{m_{1}n_{1}}\right)} \\
  &+ \sum_{m_{2}n_{2}\geq 2m_{1}n_{1}}
    \sum_{\substack{1 \leq m_1 \leq n_1 - 1 \\ 2 \leq n_1 \leq T}}
    \sum_{\substack{1 \leq m_2 \leq n_2 - 1 \\ 2 \leq n_2 \leq T}}
    n_1^{-\sigma_2} m_1^{-\sigma_1}n_2^{-\sigma_2} m_2^{-\sigma_1}
    \frac{1}{\log \left(\frac{m_{2}n_{2}}{m_{1}n_{1}}\right)} \\
  &= B_{1} + B_{2},
\end{align*}
say. In order to evaluate $B_2$, we write
\begin{align*}
  B_{2} &\ll \sum_{m_{2}n_{2}\geq 2m_{1}n_{1}}
    \sum_{\substack{1 \leq m_1 \leq n_1 - 1 \\ 2 \leq n_1 \leq T}}
    \sum_{\substack{1 \leq m_2 \leq n_2 - 1 \\ 2 \leq n_2 \leq T}}
    n_1^{-\sigma_2} m_1^{-\sigma_1}n_2^{-\sigma_2} m_2^{-\sigma_1} \\
  &= \sum_{\substack{m_{2}n_{2}\geq 2m_{1}n_{1}\\ 2m_{1}n_{1}\leq T}}
    \sum_{\substack{1 \leq m_1 \leq n_1 - 1 \\ 2 \leq n_1 \leq T}}
    \sum_{\substack{1 \leq m_2 \leq n_2 - 1 \\ 2 \leq n_2 \leq T}}
    n_1^{-\sigma_2} m_1^{-\sigma_1}n_2^{-\sigma_2} m_2^{-\sigma_1} \\
  &+ \sum_{\substack{m_{2}n_{2}\geq 2m_{1}n_{1}\\ T<2m_{1}n_{1}\leq 2T}}
    \sum_{\substack{1 \leq m_1 \leq n_1 - 1 \\ 2 \leq n_1 \leq T}}
    \sum_{\substack{1 \leq m_2 \leq n_2 - 1 \\ 2\leq n_2\leq T} }
    n_1^{-\sigma_2} m_1^{-\sigma_1}n_2^{-\sigma_2} m_2^{-\sigma_1} \\
  &+ \sum_{\substack{m_{2}n_{2}\geq 2m_{1}n_{1}\\ m_{1}n_{1}>T}}
    \sum_{\substack{1 \leq m_1 \leq n_1 - 1 \\ 2 \leq n_1 \leq T}}
    \sum_{\substack{1 \leq m_2 \leq n_2 - 1 \\ 2\leq n_2\leq T} }
    n_1^{-\sigma_2} m_1^{-\sigma_1}n_2^{-\sigma_2} m_2^{-\sigma_1} \\
  &= C_{1}+C_{2}+C_{3},
\end{align*}
say. Now, let $j \in \mathbb{N}$. If $j \le T$ then we have
\[ \sum_{\substack{m_2 n_2 = j \\
     1\leq m_{2}\leq n_{2}-1\\ 2\leq n_{2}\leq T}}
     m_{2}^{-\sigma_{1}}n_{2}^{-\sigma_{2}} =
     \sum_{\substack{m_{2}|j\\ m_{2}<\sqrt{j}}}
     m_{1}^{-\sigma_{1}+\sigma_{2}}j^{-\sigma_{2}}
     \ll j^{-\sigma_{2}+\epsilon} \]
and if $T < j < T^2$ then we have
\begin{align*}
  \sum_{\substack{m_2 n_2 = j \\ 1\leq m_{2}n_{2}\leq n_{2}-1\\
    2\leq n_{2}\leq T}} m_{2}^{-\sigma_{1}}n_{2}^{-\sigma_{2}}
    &= \sum_{\substack{m_{2}|j\\ \frac{j}{T}\leq m_{2}<\sqrt{j}}}
    m_{2}^{-\sigma_{1}+\sigma_{2}}j^{-\sigma_{2}} \\
  &\ll \left(\frac{j}{T}\right)^{-\sigma_{1}+\sigma_{2}}
    j^{-\sigma_{2}+\epsilon} \\
  &= j^{-\sigma_{1}+\epsilon}T^{\sigma_{1}-\sigma_{2}}.
\end{align*}
From these evaluations, if we set
\begin{equation}  \label{eq_lambda}
\lambda = \begin{cases}
            1 - \sigma_2 + \epsilon & (\sigma_1 > 1), \\
            2 - \sigma_1 - \sigma_2 + \epsilon & (\sigma_1 \le 1)
          \end{cases}
\end{equation}
then we obtain
\begin{align*}
  C_{1} &\ll \sum_{2\leq k \leq \frac{T}{2}}
    \sum_{\substack{m_{1}|k\\ m_{1}<\sqrt{k}}}
    m_{1}^{-\sigma_{1}+\sigma_{2}}k^{-\sigma_{2}}
    \left(\sum_{2k\leq j_{1}\leq T}j_{1}^{-\sigma_{2}+\epsilon}
    +\sum_{T<j_{2}<T^{2}}
    j_{2}^{-\sigma_{1}+\epsilon}T^{\sigma_{1}-\sigma_{2}}\right)\\
  &\ll \sum_{2\leq k \leq \frac{T}{2}}
    \sum_{\substack{m_{1}|k\\ m_{1}<\sqrt{k}}}
    m_{1}^{-\sigma_{1}+\sigma_{2}}k^{-\sigma_{2}}T^{\lambda}\\
  &\ll T^{\lambda} \sum_{2\leq k\leq \frac{T}{2}}
    k^{-\sigma_{2}+\epsilon} \\
  &\ll \begin{cases}
         T^{2-2\sigma_{2}+\epsilon} & (\sigma_1 > 1), \\
         T^{3 - \sigma_1 - 2 \sigma_2 + \epsilon} & (\sigma_1 \le 1).
       \end{cases}
\end{align*}
By the same argument, we obtain
\begin{align*}
  C_{2} &\ll \sum_{\frac{T}{2}<k\leq T}
    \sum_{\substack{m_{1}|k \\ m_{1}<\sqrt{k}}}
    m_{1}^{-\sigma_{1}+\sigma_{2}}k^{-\sigma_{2}}
    \sum_{T<j<T^{2}}j^{-\sigma_{1}+\epsilon}T^{\sigma_{1}-\sigma_{2}} \\
  &\ll \begin{cases}
         T^{2-2\sigma_{2}+\epsilon} & (\sigma_1 > 1), \\
         T^{3 - \sigma_1 - 2 \sigma_2 + \epsilon} & (\sigma_1 \le 1).
       \end{cases}
\end{align*}
If $\lambda$ is the same as in (\ref{eq_lambda}) then we obtain
\begin{align*}
  C_{3} &\ll \sum_{T<k<T^{2}}
    \sum_{\substack{m_{1}|k\\ m_{1}<\sqrt{k}\\ \frac{k}{T} \le m_{1}}}
    m_{1}^{-\sigma_{1}+\sigma_{2}}k^{-\sigma_{2}}
    \sum_{T<j<T^{2}}j^{-\sigma_{1}+\epsilon}T^{\sigma_{1}-\sigma_{2}}\\
  &\ll T^{\lambda}\sum_{T<k<T^{2}}k^{-\sigma_{2}}
    \sum_{\substack{m_{1}|k\\ m_{1}<\sqrt{k}\\ \frac{k}{T} \le m_{1}}}
    m_{1}^{-\sigma_{1}+\sigma_{2}}\\
  &\ll T^{\lambda + \sigma_1 - \sigma_2}
    \sum_{T<k<T^{2}}k^{-\sigma_{1} + \epsilon} \\
  &\ll \begin{cases}
         T^{2-2\sigma_{2}+\epsilon} & (\sigma_1 > 1), \\
         T^{4 - 2 \sigma_1 - 2 \sigma_2 + \epsilon} & (\sigma_1 \le 1).
       \end{cases}
\end{align*}
Hence, we have
\[ B_2 \ll
     \begin{cases}
       T^{2-2\sigma_{2}+\epsilon} & (\sigma_1 > 1), \\
       T^{4 - 2 \sigma_1 - 2 \sigma_2 + \epsilon} & (\sigma_1 \le 1).
     \end{cases} \]
In order to evaluate $B_1$, we rewrite
\begin{align*}
  B_1 &= \sum_{\substack{m_1 n_1 < m_2 n_2<2m_{1}n_{1}\\ 2m_{1}n_{1}<T}}
    \sum_{\substack{1 \leq m_1 \leq n_1 - 1 \\ 2 \leq n_1 \leq T}}
    \sum_{\substack{1 \leq m_2 \leq n_2 - 1 \\ 2 \leq n_2 \leq T}}
    n_1^{-\sigma_2} m_1^{-\sigma_1}n_2^{-\sigma_2} m_2^{-\sigma_1}
    \frac{1}{\log \left(\frac{m_{2}n_{2}}{m_{1}n_{1}}\right)} \\
  &+ \sum_{\substack{m_1 n_1 < m_2 n_2<2m_{1}n_{1}\\
    \frac{T}{2}\leq m_{1}n_{1} \le T}}
    \sum_{\substack{1 \leq m_1 \leq n_1 - 1 \\ 2 \leq n_1 \leq T}}
    \sum_{\substack{1 \leq m_2 \leq n_2 - 1 \\ 2 \leq n_2 \leq T}}
    n_1^{-\sigma_2} m_1^{-\sigma_1}n_2^{-\sigma_2} m_2^{-\sigma_1}
    \frac{1}{\log \left(\frac{m_{2}n_{2}}{m_{1}n_{1}}\right)} \\
  &+ \sum_{\substack{m_1 n_1 < m_2 n_2<2m_{1}n_{1} \\ m_1 n_1 >T}}
    \sum_{\substack{1 \leq m_1 \leq n_1 - 1 \\ 2 \leq n_1 \leq T}}
    \sum_{\substack{1 \leq m_2 \leq n_2 - 1 \\ 2 \leq n_2 \leq T}}
    n_1^{-\sigma_2} m_1^{-\sigma_1}n_2^{-\sigma_2} m_2^{-\sigma_1}
    \frac{1}{\log \left(\frac{m_{2}n_{2}}{m_{1}n_{1}}\right)} \\
  &= D_{1}+D_{2}+D_{3},
\end{align*}
say. Now, we set $n_{2}=\frac{m_{1}n_{1}+r}{m_{2}}$. If $m_2 n_2 \le T$
then we have
\begin{align*}
  &\sum_{\substack{m_{2}|m_{1}n_{1}+r\\ m_{2}<\sqrt{m_{1}n_{1}+r}}}
    m_{2}^{-\sigma_{1}}n_{2}^{-\sigma_{2}}n_{1}^{-\sigma_{2}}
    m_{1}^{-\sigma_{1}}
    \frac{1}{\log\left(\frac{n_{2}m_{2}}{n_{1}m_{1}}\right)}\\
  &\ll \sum_{\substack{m_{2}|m_{1}n_{1}+r\\ m_{2}<\sqrt{m_{1}n_{1}+r}}}
    m_{2}^{\sigma_{2}-\sigma_{1}}(m_{1}n_{1}+r)^{-\sigma_{2}}
    \frac{n_{1}m_{1}}{r}m_{1}^{-\sigma_{1}}n_{1}^{-\sigma_{2}}\\
  &\ll (m_{1}n_{1}+r)^{-\sigma_{2}+\epsilon}
    \frac{m_{1}^{1-\sigma_{1}}n_{1}^{1-\sigma_{2}}}{r}
\end{align*}
and if $m_2 n_2 > T$ then we have
\begin{align*}
  &\sum_{\substack{m_{2}|m_{1}n_{1}+r\\ m_{2}<\sqrt{m_{1}n_{1}+r}\\
    \frac{m_{1}n_{1}+r}{T}\leq m_{2}}}
    m_{2}^{-\sigma_{1}}n_{2}^{-\sigma_{2}}n_{1}^{-\sigma_{2}}
    m_{1}^{-\sigma_{1}}
    \frac{1}{\log\left(\frac{n_{2}m_{2}}{n_{1}m_{1}}\right)}\\
  &\ll \sum_{\substack{m_{2}|m_{1}n_{1}+r\\ m_{2}<\sqrt{m_{1}n_{1}+r}\\
    \frac{m_{1}n_{1}+r}{T}\leq m_{2}}}
    m_{2}^{-\sigma_{1}}
    \left(\frac{m_{1}n_{1}+r}{m_{2}}\right)^{-\sigma_{2}}
    \frac{m_{1}^{1-\sigma_{1}}n_{1}^{1-\sigma_{2}}}{r}\\
  &\ll \left(\frac{m_{1}n_{1}+r}{T}\right)^{\sigma_{2}-\sigma_{1}}
    (m_{1}n_{1}+r)^{-\sigma_{2}+\epsilon}
    \frac{m_{1}^{1-\sigma_{1}}n_{1}^{1-\sigma_{2}}}{r}.
\end{align*}
From these evaluations and Remark \ref{rem_sumint}, we have
\begin{align*}
  D_{1} &\ll \sum_{2\leq k < \frac{T}{2}}
    \sum_{\substack{m_{1}|k\\m_{1}<\sqrt{k}}}
    \sum_{r=1}^{k}(k+r)^{-\sigma_{2}+\epsilon}
    \frac{k^{1-\sigma_{2}}}{r}m_{1}^{\sigma_{2}-\sigma_{1}}\\
  &\ll \sum_{2\leq k < \frac{T}{2}}
    \sum_{\substack{m_{1}|k\\m_{1}<\sqrt{k}}}
    k^{-2\sigma_{2}+1+\epsilon}m_{1}^{\sigma_{2}-\sigma_{1}}
    \ll T^{-2\sigma_{2}+2+\epsilon}
\end{align*}
and
\begin{align*}
  D_{2} &\ll \sum_{\frac{T}{2}\leq k \le T}
    \sum_{\substack{m_{1}|k\\m_{1}<\sqrt{k}}}
    \sum_{1\leq r\leq T-k}(k+r)^{-\sigma_{2}+\epsilon}
    \frac{km_{1}^{-\sigma_{1}}}{r}
    \left(\frac{k}{m_{1}}\right)^{-\sigma_{2}}+\\
  &+ \sum_{\frac{T}{2}\leq k \le T}
    \sum_{\substack{m_{1}|k\\m_{1}<\sqrt{k}}}
    \sum_{T-k < r\leq k}\left(\frac{k+r}{T}\right)^{\sigma_{2}-\sigma_{1}}
    (k+r)^{-\sigma_{2}+\epsilon}
    \frac{m_{1}^{1-\sigma_{1}}\left(\frac{k}{m_{1}}\right)^{1-\sigma_{2}}}{r}\\
  &= E_{1}+E_{2},
\end{align*}
say. By Remark \ref{rem_sumint}, we obtain
\begin{equation*}
  E_{1} \ll \sum_{\frac{T}{2}\leq k \le T}
    \sum_{\substack{m_{1}|k\\m_{1}<\sqrt{k}}}
    k^{1-2\sigma_{2}+\epsilon}m_{1}^{\sigma_{2}-\sigma_{1}}
    \ll\sum_{\frac{T}{2}\leq k \le T}k^{1-2\sigma_{2}+\epsilon}
    \ll T^{2-2\sigma_{2}+\epsilon}
\end{equation*}
and
\begin{align*}
  E_{2} &\ll \sum_{\frac{T}{2}\leq k \le T}
    \sum_{\substack{m_{1}|k\\m_{1}<\sqrt{k}}}
    \sum_{T-k \leq r \leq k} T^{\sigma_{1}-\sigma_{2}}
    (k+r)^{-\sigma_{1}+\epsilon}m_{1}^{\sigma_{2}-\sigma_{1}}
    k^{1-\sigma_{2}}r^{-1}\\
  &\ll \sum_{\frac{T}{2}\leq k \le T}
    \sum_{\substack{m_{1}|k\\m_{1}<\sqrt{k}}}
    T^{\sigma_{1}-\sigma_{2}}T^{-\sigma_{1}+\epsilon}
    m_{1}^{\sigma_{2}-\sigma_{1}}k^{1-\sigma_{2}}\\
  &\ll T^{2-2\sigma_{2}+\epsilon}.
\end{align*}
Hence, we have $D_{2}\ll T^{2-2\sigma_{2}+\epsilon}$.
Lastly, we evaluate $D_3$. By Remark \ref{rem_sumint},
\begin{align*}
  D_{3} &= \sum_{T<m_{1}n_{1}<m_{2}n_{2}<2m_{1}n_{1}}
    \sum_{\substack{1\leq m_{1}\leq n_{1}-1\\ 2\leq n_{1}\leq T}}
    \sum_{\substack{1\leq m_{2}\leq n_{2}-1\\ 2\leq n_{2}\leq T}}
    m_{2}^{-\sigma_{1}}n_{2}^{-\sigma_{2}}n_{1}^{-\sigma_{2}}
    m_{1}^{-\sigma_{1}}
    \frac{1}{\log\left(\frac{n_{2}m_{2}}{n_{1}m_{1}}\right)}\\
  &\ll \sum_{T<m_{1}n_{1}<m_{2}n_{2}<2m_{1}n_{1}}
    \sum_{\substack{1\leq m_{1}\leq n_{1}-1\\ 2\leq n_{1}\leq T}}
    \sum_{r=1}^{m_{1}n_{1}}
    \sum_{\substack{m_{2}|m_{1}n_{1}+r\\ m_{2}<\sqrt{m_{1}n_{1}+r}\\
    \frac{m_{1}n_{1}+r}{T}<m_{2}}}
    m_{1}^{-\sigma_{1}}n_{1}^{-\sigma_{2}}m_{2}^{-\sigma_{1}}
    \left(\frac{m_{1}n_{1}+r}{m_{2}}\right)^{-\sigma_{2}}
    \frac{m_{1}n_{1}}{r}\\
  &\ll \sum_{T<m_{1}n_{1}<m_{2}n_{2}<2m_{1}n_{1}}
    \sum_{\substack{1\leq m_{1}\leq n_{1}-1\\ 2\leq n_{1}\leq T}}
    \sum_{r=1}^{m_{1}n_{1}}
    \left(\frac{m_{1}n_{1}+r}{T}\right)^{\sigma_{2}-\sigma_{1}}
    m_{1}^{1-\sigma_{1}}n_{1}^{1-\sigma_{2}}r^{-1}
    (m_{1}n_{1}+r)^{-\sigma_{2} + \epsilon}\\
  &\ll \sum_{T\leq k < T^{2}}
    \sum_{\substack{m_{1}|k\\ \frac{k}{T}<m_{1}<\sqrt{k}}}
    \sum_{r=1}^{k} T^{\sigma_{1}-\sigma_{2}}(k+r)^{-\sigma_{1} + \epsilon}
    m_{1}^{\sigma_{2}-\sigma_{1}}k^{1-\sigma_{2}}r^{-1}\\
  &\ll \sum_{T\leq k < T^{2}}
    \sum_{\substack{m_{1}|k\\ \frac{k}{T}<m_{1}<\sqrt{k}}}
    k^{-\sigma_{1}+\epsilon}T^{\sigma_{1}-\sigma_{2}}
    m_{1}^{\sigma_{2}-\sigma_{1}}k^{1-\sigma_{2}}\\
  &\ll \sum_{T\leq k < T^{2}}
    k^{1-\sigma_{1}-\sigma_{2}+\epsilon}
    \left(\frac{k}{T}\right)^{\sigma_{2}-\sigma_{1}}T^{\sigma_{1}-\sigma_{2}}\\
  &\ll
    \begin{cases}
      T^{2-2\sigma_{2}+\epsilon} & (\sigma_1 > 1), \\
      T^{4 - 2 \sigma_1 - 2 \sigma_2 + \epsilon} & (\sigma_1 \le 1).
    \end{cases}
\end{align*}
Hence, we have
\[ B_1 \ll
    \begin{cases}
      T^{2-2\sigma_{2}+\epsilon} & (\sigma_1 > 1), \\
      T^{4 - 2 \sigma_1 - 2 \sigma_2 + \epsilon} & (\sigma_1 \le 1).
    \end{cases} \]
This implies
\[ S_3 \ll
    \begin{cases}
      T^{2-2\sigma_{2}+\epsilon} & (\sigma_1 > 1), \\
      T^{4 - 2 \sigma_1 - 2 \sigma_2 + \epsilon} & (\sigma_1 \le 1).
    \end{cases} \]
Therefore we have
\[ \int_{2}^{T}|A(s_{1},s_{2})|^{2}dt = \zeta_2^{\Box}(\sigma_1, \sigma_2)T +
   \begin{cases}
     O(T^{2-2\sigma_{2}+\epsilon}) & (\sigma_1 >1), \\
     O(T^{4 - 2 \sigma_1 - 2 \sigma_2 + \epsilon}) & (\sigma_1 \le 1).
   \end{cases} \]
Now, if we set
\[ \lambda =
     \begin{cases}
       - \sigma_2 & (\sigma_1 > 1), \\
       - \sigma_2  + \epsilon & (\sigma_1 = 1), \\
       1 - \sigma_1 - \sigma_2 & (\sigma_1 < 1)
     \end{cases} \]
then we have
\begin{align*}
  \int_{2}^{T}|\zeta(s_{1},s_{2})|^{2}dt &=
    \int_{2}^{T}|A(s_{1},s_{2})+O(t^{\lambda})|^{2}dt\\
  &= \int_{2}^{T}|A(s_{1},s_{2})|^{2}dt +
    O\left(\int_{2}^{T}|A(s_{1},s_{2})t^{\lambda}|dt\right) + O(1).
\end{align*}
By the Cauchy-Schwarz inequality, we have
\begin{align*}
  \int_{2}^{T}|A(s_{1},s_{2})t^{\lambda}|dt
  &\ll \left(\int_{2}^{T}|A(s_{1},s_{2})|^{2}dt\right)^{\frac{1}{2}}
    \left(\int_{2}^{T}t^{2 \lambda}dt\right)^{\frac{1}{2}}\\
  &\ll T^{\frac{1}{2}}.
\end{align*}
This implies the theorems.
\end{proof}

Next, we consider the case $\sigma_1 \le \sigma_2$.

\begin{proof}[Proof of Theorem \ref{th_intt} for $\sigma_1 \le \sigma_2$]
Hereafter we use the same notations as in the previous proof.
First we evaluate $S_1$. Since
\begin{align*}
  \sum_{k > T} \biggl( \sum_{\substack{mn = k \\ m<n}}
    m^{-\sigma_1} n^{-\sigma_2} \biggr)^2
  &= \sum_{k > T} \biggl( \sum_{\substack{m|k \\ m < \sqrt{k}}}
    m^{-\sigma_1} m^{\sigma_2} k^{-\sigma_2} \biggr)^2 \\
  &\ll \sum_{k > T} k^{-2 \sigma_2} \bigl( k^{\frac{1}{2}
    (\sigma_2 - \sigma_1) + \epsilon} \bigr)^2 \\
  &\ll \sum_{k > T} k^{- \sigma_1 - \sigma_2 + \epsilon}
    \ll T^{1 - \sigma_1 - \sigma_2 + \epsilon},
\end{align*}
we have
\[  S_1 = \zeta_2^{\Box}(\sigma_1, \sigma_2)
          + O(T^{1 - \sigma_1 - \sigma_2 + \epsilon}). \]
Next we evaluate $S_2$. Since
\begin{equation}  \label{ineq_divf}
  \sum_{\substack{mn = k \\ m < n}} m^{-\sigma_1} n^{-\sigma_2}
  = \sum_{\substack{m \mid k \\ m < \sqrt{k}}}
    m^{\sigma_2 - \sigma_1} k^{- \sigma_2}
  \ll k^{- \frac{1}{2}(\sigma_1 + \sigma_2) + \epsilon},
\end{equation}

\begin{align*}
  \sum_{\substack{mn = k \\ m < n}} m^{-\sigma_1} n^{1- \sigma_2}
  &= \sum_{\substack{m | k \\ m < \sqrt{k}}}
    k^{1 - \sigma_2} m^{\sigma_2 - \sigma_1 - 1} 
  &\ll \begin{cases}
        k^{1 - \sigma_2 + \epsilon} & (\sigma_2 - \sigma_1 - 1 \le 0) \\
        k^{\frac{1}{2}(1 - \sigma_1 - \sigma_2) + \epsilon}
          & (\sigma_2 - \sigma_1 - 1 > 0)
      \end{cases}
\end{align*}
hold, we have
\begin{align*}
  A_1, A_2 &\ll
    \begin{cases}
      \displaystyle{\sum_{2 \le k \le T}
        k^{- \frac{1}{2}(\sigma_1 + \sigma_2) + \epsilon}
        k^{1 - \sigma_2 + \epsilon}} & (\sigma_2 - \sigma_1 - 1 \le 0) \\
      \displaystyle{\sum_{2 \le k \le T}
        k^{- \frac{1}{2}(\sigma_1 + \sigma_2) + \epsilon}
        k^{\frac{1}{2}(1 - \sigma_1 - \sigma_2) + \epsilon}}
        & (\sigma_2 - \sigma_1 - 1 > 0)
    \end{cases} \\
    &= \begin{cases}
         \displaystyle{\sum_{2 \le k \le T}
           k^{1- \frac{1}{2} \sigma_1 - \frac{3}{2} \sigma_2 + \epsilon}}
           & (\sigma_2 - \sigma_1 - 1 \le 0) \\
         \displaystyle{\sum_{2 \le k \le T}
           k^{\frac{1}{2} - \sigma_1 - \sigma_2 + \epsilon}}
           & (\sigma_2 - \sigma_1 - 1 > 0).
       \end{cases}
\end{align*}
We note that
$1- \frac{1}{2} \sigma_1 - \frac{3}{2} \sigma_2 < -1$ is equivalent to
$\sigma_2 > - \frac{1}{3} \sigma_1 + \frac{4}{3}$. Hence we have
\[ A_1, A_2 \ll
     \begin{cases}
       T^{2 - \frac{1}{2} \sigma_1 - \frac{3}{2} \sigma_2 + \epsilon}
       & (\text{$\sigma_2 - \sigma_1 - 1 \le 0$ and
         $\sigma_2 \le - \frac{1}{3} \sigma_1 + \frac{4}{3}$}), \\
       1 & (\text{otherwise})
     \end{cases} \]
because $\sigma_1 + \sigma_2 > 3/2$. Similarly, we have
\begin{align*}
  A_3, A_4 &\ll
    \begin{cases}
      \displaystyle{\sum_{T < k < T^2}
        k^{1- \frac{1}{2} \sigma_1 - \frac{3}{2} \sigma_2 + \epsilon}}
        & (\sigma_2 - \sigma_1 - 1 \le 0)  \\
      \displaystyle{\sum_{T < k < T^2}
        k^{\frac{1}{2} - \sigma_1 - \sigma_2 + \epsilon}}
        & (\sigma_2 - \sigma_1 - 1 > 0)
    \end{cases} \\
    &\ll \begin{cases}
           T^{4 - \sigma_1 - 3 \sigma_2 + \epsilon}
           & (\text{$\sigma_2 - \sigma_1 - 1 \le 0$ and
             $\sigma_2 \le - \frac{1}{3} \sigma_1 + \frac{4}{3}$}) \\
           1 & (\text{otherwise}).
         \end{cases}
\end{align*}
Therefore we have
\[ S_2 \ll
     \begin{cases}
       T^{4 - \sigma_1 - 3 \sigma_2 + \epsilon}
       & (\text{$\sigma_2 - \sigma_1 - 1 \le 0$ and
         $\sigma_2 \le - \frac{1}{3} \sigma_1 + \frac{4}{3}$}) \\
       1 & (\text{otherwise}).
     \end{cases} \]
Next we evaluate $S_3$. If we set $m_2 n_2 = j$, by (\ref{ineq_divf})
and $-1 \le - \frac{1}{2}(\sigma_1 + \sigma_2) < - 3/4$,
we have
\begin{align*}
  C_1 &\ll \sum_{2 \le k \le \frac{T}{2}}
    \sum_{\substack{m_1 | k \\ m_1 < \sqrt{k}}}
      m_{1}^{-\sigma_{1}+\sigma_{2}}k^{-\sigma_{2}}
    \left( \sum_{2k \le j < T^2}
      j^{-\frac{1}{2}(\sigma_1 + \sigma_2) + \epsilon} \right) \\
  &\ll T^{2 - \sigma_1 - \sigma_2 + \epsilon}
    \sum_{2 \le k \le \frac{T}{2}}
      k^{- \frac{1}{2}(\sigma_1 + \sigma_2) + \epsilon} \\
  &\ll T^{3 - \frac{3}{2}(\sigma_1 + \sigma_2) + \epsilon}.
\end{align*}
Similarly, we have
$C_2 \ll T^{3 - \frac{3}{2}(\sigma_1 + \sigma_2) + \epsilon}$ and
\begin{align*}
  C_3 &\ll \sum_{T < k < T^2}
    \sum_{\substack{m_1 | k \\ \frac{k}{T} \le m_1 < \sqrt{k}}}
      m_{1}^{-\sigma_{1}+\sigma_{2}}k^{-\sigma_{2}}
    \left( \sum_{T \le j < T^2}
      j^{-\frac{1}{2}(\sigma_1 + \sigma_2) + \epsilon} \right) \\
  &\ll T^{2 - \sigma_1 - \sigma_2 + \epsilon}
    \sum_{T < k < T^2}
      k^{- \frac{1}{2}(\sigma_1 + \sigma_2) + \epsilon} \\
  &\ll T^{4 - 2 \sigma_1 - 2 \sigma_2 + \epsilon}.
\end{align*}
Since $4 - 2 \sigma_1 - 2 \sigma_2 - (3 - \frac{3}{2}(\sigma_1 + \sigma_2))
= 1 - \frac{1}{2}(\sigma_1 + \sigma_2) \ge 0$, we see that
$B_2 \ll T^{4 - 2 \sigma_1 - 2 \sigma_2 + \epsilon}$ holds.
Now, we set $m_2 n_2 = m_1 n_1 + r$ ($r \in \{1,2, \dots ,m_1 n_1 -1 \}$).
Since $x \asymp \log (1+x)$ for $x \in [0,1]$, we have
\begin{align*}
  &\sum_{\substack{m_{2}|m_{1}n_{1}+r\\ m_{2}<\sqrt{m_{1}n_{1}+r}}}
    m_{2}^{-\sigma_{1}}n_{2}^{-\sigma_{2}}n_{1}^{-\sigma_{2}}
    m_{1}^{-\sigma_{1}}
    \frac{1}{\log\left(\frac{n_{2}m_{2}}{n_{1}m_{1}}\right)}\\
  &\ll \sum_{\substack{m_{2}|m_{1}n_{1}+r\\ m_{2}<\sqrt{m_{1}n_{1}+r}}}
    m_{2}^{\sigma_{2}-\sigma_{1}}(m_{1}n_{1}+r)^{-\sigma_{2}}
    \frac{n_{1}m_{1}}{r}m_{1}^{-\sigma_{1}}n_{1}^{-\sigma_{2}}\\
  &\ll (m_1 n_1 + r)^{- \frac{1}{2}(\sigma_1 + \sigma_2) + \epsilon}
    \frac{m_1^{1- \sigma_1} n_1^{1- \sigma_2}}{r}.
\end{align*}
From this evaluation and Remark \ref{rem_sumint}, we have
\begin{align*}
  D_1 &= \sum_{\substack{m_1 n_1 = k \\ 2 \le k < \frac{T}{2} \\
    m_1 < n_1}} \sum_{r = 1}^{k-1}
    \sum_{\substack{m_{2}|m_{1}n_{1}+r\\ m_{2}<\sqrt{m_{1}n_{1}+r}}}
    m_{2}^{-\sigma_{1}}n_{2}^{-\sigma_{2}}n_{1}^{-\sigma_{2}}
    m_{1}^{-\sigma_{1}}
    \frac{1}{\log\left(\frac{n_{2}m_{2}}{n_{1}m_{1}}\right)}\\
  &\ll \sum_{2\leq k < \frac{T}{2}}
    \sum_{\substack{m_{1}|k\\m_{1}<\sqrt{k}}}
    \sum_{r=1}^{k} (k+r)^{- \frac{1}{2}(\sigma_1 + \sigma_2) + \epsilon}
    \frac{k^{1- \sigma_2}}{r} m_1^{\sigma_2 - \sigma_1} \\
  &\ll \sum_{2\leq k < \frac{T}{2}}
    \sum_{\substack{m_{1}|k\\m_{1}<\sqrt{k}}}
    m_1^{\sigma_2 - \sigma_1}
    k^{1 - \frac{1}{2} \sigma_1 - \frac{3}{2} \sigma_2 + \epsilon} \\
  &\ll \sum_{2\leq k < \frac{T}{2}}
    k^{\frac{1}{2}(\sigma_2 - \sigma_1) + \epsilon}
    k^{1 - \frac{1}{2} \sigma_1 - \frac{3}{2} \sigma_2 + \epsilon} \\
  &= \sum_{2\leq k < \frac{T}{2}}
    k^{1- \sigma_1 - \sigma_2 + \epsilon} \\
  &\ll T^{2 - \sigma_1 - \sigma_2}.
\end{align*}
Similarly, we have $D_2 \ll T^{2 - \sigma_1 - \sigma_2}$ and
\begin{align*}
  D_3 &= \sum_{\substack{m_1 n_1 = k \\ T < k < T^2 \\
    m_1 < n_1}} \sum_{r = 1}^{k-1}
    \sum_{\substack{m_{2}|m_{1}n_{1}+r\\ m_{2}<\sqrt{m_{1}n_{1}+r}}}
    m_{2}^{-\sigma_{1}}n_{2}^{-\sigma_{2}}n_{1}^{-\sigma_{2}}
    m_{1}^{-\sigma_{1}}
    \frac{1}{\log\left(\frac{n_{2}m_{2}}{n_{1}m_{1}}\right)}\\
  &\ll \sum_{T < k < T^2} k^{1- \sigma_1 - \sigma_2 + \epsilon} \\
  &\ll T^{4 - 2 \sigma_1 - 2 \sigma_2 + \epsilon}.
\end{align*}
Therefore we have $B_1 \ll T^{4 - 2 \sigma_1 - 2 \sigma_2 + \epsilon}$,
and we have $S_3 \ll T^{4 - 2 \sigma_1 - 2 \sigma_2 + \epsilon}$.
Since $4 - 2 \sigma_1 - 2 \sigma_2 - (4 - \sigma_1 - 3 \sigma_2) =
\sigma_2 - \sigma_1 \ge 0$, we have
\[ \int_{2}^{T}|A(s_{1},s_{2})|^{2}dt = \zeta_2^{\Box}(\sigma_1, \sigma_2)T +
     O(T^{4 - 2 \sigma_1 - 2 \sigma_2 + \epsilon}). \]
By the same argument as in the case $\sigma_1 > \sigma_2$, we obtain
the theorem.
\end{proof}

\begin{proof}[Proof of Theorem \ref{th_intt} for $\sigma_1 > 1$
and $\sigma_2 = 1/2$]
Since
\[ \frac{1}{|\zeta(\sigma)|} \le | \zeta(\sigma + it) | \le
   \zeta(\sigma) \]
for $\sigma > 1$ (see \cite{titchmarsh}), we have
\[ \int_2^T |\zeta(s_1)\zeta(s_2)|^2 dt \asymp T \log T. \]
We have proved
\[ \int_2^T |\zeta_2(s_2, s_1)|^2 dt = O(T) \]
in this section. From these evaluations, we can obtain
\[ I^{\Box}(T) \asymp T \log T \]
by the same argument as in the proof of Theorem \ref{th_intt2}
for $\sigma_1 > 1$ and $\sigma_2 = 1/2$.
\end{proof}

\end{document}